\documentstyle[amssymb,amsfonts]{amsart}

\def\normo#1{\left\|#1\right\|}
\def\normb#1{\big\|#1\big\|}

\def\aabs#1{\left|#1\right|}
\def\brk#1{\left(#1\right)}

\def\rev#1{\frac{1}{#1}}

\def\norm#1{\|#1\|}
\def\jb#1{\langle#1\rangle}

\newcommand{\R}{{\mathbb R}}
\newcommand{\C}{{\mathbb C}}
\newcommand{\Z}{{\mathbb Z}}
\newcommand{\ft}{{\mathcal{F}}}
\newcommand{\Hl}{{\mathcal{H}}}

\newcommand{\les}{{\lesssim}}
\newcommand{\ges}{{\gtrsim}}
\newcommand{\ra}{{\rightarrow}}

\newcommand{\G}{{\mathcal{G}}}
\newcommand{\Sch}{{\mathcal{S}}}

\newcommand{\sgn}{{\mbox{sgn}}}
\newcommand{\sub}{{\mbox{sub}}}
\newcommand{\thd}{{\mbox{thd}}}
\newcommand{\U}{{U_\delta}}

\newcommand{\dr}{{\omega_\delta}}
\numberwithin{equation}{section}

\theoremstyle{plain}
  \newtheorem{theorem}[subsection]{Theorem}
  \newtheorem{proposition}[subsection]{Proposition}
  \newtheorem{lemma}[subsection]{Lemma}
  \newtheorem{corollary}[subsection]{Corollary}

\theoremstyle{remark}
  \newtheorem{remark}[subsection]{Remark}

\theoremstyle{definition}

\newenvironment{proof}{\noindent {\bf Proof.} }{\endprf\par}
\def \endprf{\hfill  {\vrule height6pt width6pt depth0pt}\medskip}

\begin{document}
\title[modified finite-depth-fluid equation]{Global well-posedness and limit behavior for the modified finite-depth-fluid equation}

\author{Zihua Guo and Baoxiang Wang}
\address{LMAM, School of Mathematical Sciences, Peking University, Beijing
100871, China}

\email{zihuaguo, wbx@@math.pku.edu.cn}

\urladdr{http://guo.5188.org}

\begin{abstract} Considering the Cauchy problem for the modified
finite-depth-fluid equation
\begin{eqnarray*}
\partial_tu-\G_\delta(\partial_x^2u)\mp u^2u_x=0,\ \
u(0)=u_0,
\end{eqnarray*}
where $\G_\delta f=-i \ft ^{-1}[\coth(2\pi \delta \xi)-\frac{1}{2\pi
\delta \xi}]\ft f$, $\delta\ges 1$, and $u$ is a real-valued
function, we show that it is uniformly globally well-posed if $u_0
\in H^s\ (s\geq 1/2)$ with $\norm{u_0}_{L^2}$ sufficiently small for
all $\delta \ges 1$. Our result is sharp in the sense that the
solution map fails to be $C^3$ in $H^s (s<1/2)$. Moreover, we prove
that for any $T>0$, its solution converges in $C([0,T]; \,H^s)$ to
that of the modified Benjamin-Ono equation if $\delta$ tends to
$+\infty$.
\end{abstract}



\maketitle


\section{Introduction}
In this paper, we study the Cauchy problem for the (defocusing)
modified finite-depth-fluid (mFDF) equation (the focusing version
with nonlinearity $u^2u_x$ can also be treated by our methods)
\begin{eqnarray}\label{eq:mFDF}
&&\partial_tu-\G_\delta(\partial_x^2u)- u^2u_x=0,\ u(x,0)=u_0(x),
\end{eqnarray}
where $u: \R^2 \ra \R$ is a real-valued function of $(x,t)\in
\R\times \R$,
\begin{equation}\label{eq:Gdelta}
\G_\delta f=-i \ft ^{-1}[\coth(2\pi \delta \xi)-\frac{1}{2\pi \delta
\xi}]\ft f,
\end{equation}
and $\delta> 0$ is a real number which characterizes the depth of
the fluid layer. The equation \eqref{eq:mFDF} is a special one of
the following so-called generalized finite-depth-fluid equations
\begin{eqnarray}\label{eq:gFDF}
\partial_tu-\G_\delta(\partial_x^2u)+u^{k}u_x=0,\ u(x,0)=u_0(x).
\end{eqnarray}

Eq. \eqref{eq:gFDF} with $k=1$ was first derived by Joseph
\cite{Jos,KSA,SAK} to describe the propagation of internal waves in
the stratified fluid of finite depth. From the physical point of
view, if the depth $\delta$ tends to infinity, then Eq.
\eqref{eq:mFDF} reduces to the modified Benjamin-Ono equation
\begin{eqnarray}\label{eq:mBO}
\partial_t u-\Hl(\partial_x^2 u)-u^2u_x=0, \ u(x,0)=u_0(x),
\end{eqnarray}
where $\Hl=-i\ft^{-1}\sgn(\xi)\ft$ denotes the Hilbert transform.
There is another form of the modified finite-depth-fluid equation
which is
\begin{eqnarray}\label{eq:mFDF2}
\partial_tu-\frac{3}{2\pi \delta}\G_\delta(\partial_x^2u)-u^2u_x=0,\ u(x,0)=u_0(x).
\end{eqnarray}
It is easy to see that under the transformation
\begin{eqnarray}\label{eq:trans}
u(t,x)\rightarrow \brk{\frac{3}{2\pi \delta}}^{1/2}u(\frac{3}{2\pi
\delta}t,x),
\end{eqnarray}
Eq. \eqref{eq:mFDF} turns into Eq. \eqref{eq:mFDF2}. If the depth
$\delta$ tends to $0$, then Eq. \eqref{eq:mFDF2} becomes the
modified Korteweg-de Vries equation
\begin{eqnarray}\label{eq:mKdV}
\partial_t u+\partial_x^3 u-u^2u_x=0.
\end{eqnarray}

There are a few literatures which are concerned with the
wellposedness for the Cauchy problem \eqref{eq:gFDF}. For the case
$k=1$, using the energy methods, Abdelouhab, Bona, Felland and Saut
\cite{ABFS} obtained global wellposedness in $H^s$ with $s>3/2$, and
the limit behavior as $\delta \rightarrow \infty$ and
$\delta\rightarrow 0$ of the solutions of Eqs. \eqref{eq:gFDF} in
$C^k([0,T];H^{s-2k})\ (s>3/2)$ and $C([0,T];H^s)\ (s\geq 2)$. For
$k\geq 4$, Han and Wang \cite{HWang} proved global wellposedness for
the equation \eqref{eq:gFDF} with small initial data in the critical
Besov spaces by using the smoothing effect estimates. To the
authors' knowledge, we are not aware of any other wellposedness
results. On the other hand, the limit equations \eqref{eq:mBO} and
\eqref{eq:mKdV} have been extensively studied during the past
decades. See \cite{Tao} for a thorough review.

In the first part of this paper, we study the wellposedness for the
Cauchy problem \eqref{eq:mFDF}. Our methods are inspired by the
important observation made by the first-named author \cite{Guo} for
the modified Benjamin-Ono equation. Precisely, for the modified
Benjamin-Ono equation, one may use a direct contraction principle to
prove wellposedness but without using a gauge transformation. We
will adopt the same ideas for the mFDF equation. From the technical
point of view, Eq. \eqref{eq:mFDF} is easier to handle than Eq.
\eqref{eq:mFDF2}. Indeed, to prove wellposedness by iteration, the
biggest enemy is the loss of derivative from the nonlinearity and
the worst case is the high-low interaction. We will see from Lemma
\ref{elemes} that if $\delta\ges 1$, the dispersion relation of Eq.
\eqref{eq:mFDF} has uniform estimates in high frequency while that
of Eq. \eqref{eq:mFDF2} doesn't. Our methods rely heavily on the
symmetries of the mFDF equation \eqref{eq:mFDF}. The first one is
the scaling invariance which enables us to assume the initial data
has small norm. It is easy to see that Eq. \eqref{eq:mFDF} is
invariant under the following transformation
\begin{equation}\label{eq:scaling}
u(x,t)\ra\
u_\lambda=\frac{1}{\lambda^{1/2}}u(\frac{x}{\lambda},\frac{t}{{\lambda^2}}),
\ u_0 \rightarrow
u_{0,\lambda}=\frac{1}{\lambda^{1/2}}u_0(\frac{x}{\lambda}), \
\delta \rightarrow \lambda\delta.
\end{equation}
We will assume $\lambda\gg 1$, thus $\lambda \delta \ges 1$ if
$\delta \ges 1$. There are at least the following three conservation
laws preserved under the flow of \eqref{eq:mBO}
\begin{eqnarray}\label{eq:conservation}
&&\frac{d}{dt}\int_\R u(x,t)dx=0, \\
&&\frac{d}{dt}\int_\R u(x,t)^2dx=0,\label{eq:L2con}\\
&&\frac{d}{dt}\int_\R \frac{1}{2}u\G_\delta u_x-\frac{1}{12}
u(x,t)^4dx=0.\label{eq:H1half}
\end{eqnarray}
These conservation laws provide a priori bounds on the solution. For
example, we can get from Lemma \ref{elemes}, \eqref{eq:L2con} and
\eqref{eq:H1half} that if $u$ is a smooth solution to
\eqref{eq:mFDF} (for the focusing case, we assume
$\norm{u_0}_{L^2}\ll 1$) and $\delta\ges 1$ then
\begin{eqnarray}\label{eq:H1halfpriori}
\norm{u}_{H^{1/2}}\les C(\norm{u_0}_{H^{1/2}}).
\end{eqnarray}

There are several methods to compensate the loss of derivative from
the nonlinearity. Energy methods exploit the "energy cancelation",
which usually requires high regularity of the initial data. Another
approach is the smoothing effect estimate for the linear solution.
On the other hand, Bourgain's space $X^{s,b}$ defined as a closure
of the following space
\[\{f\in \Sch(\R^2):\norm{f}_{X^{s,b}}=\norm{\jb{\xi}^s\jb{\tau-\dr(\xi)}^b\widehat{f}(\xi,\tau)}_{L^2}\}\]
is very useful in the study of the low regularity theory of the
nonlinear dispersive equations \cite{Bour, KPV4, IKT}. One might try
a direct perturbative approach in $X^{s,b}$ space as Kenig, Ponce
and Vega \cite{KPV4} did for the KdV and modified KdV equations.
However, one will find that the key trilinear estimate
\begin{equation}\label{eq:trilinearX}
\norm{\partial_x(u^3)}_{X^{s,b-1}}\les \norm{u}_{X^{s,b}}^3, \
\mbox{ for some } b\in [1/2,1)
\end{equation}
 fails for any $s$ due to logarithmic
divergences involving the modulation variable (see Proposition
\ref{countertrilinear}, \ref{counterXsb} below). We found that these
logarithmic divergences can be removed by us using Banach spaces
which combine $X^{s,b}$ structure with smoothing effect structure as
we found for the mBO equation \cite{Guo}. However, compared to the
mBO equation, there is a new difficulty caused by the component
$\xi/\delta$ in the dispersion relation. Fortunately, there is a
cancelation we can use. Precisely, the resonance is almost the same
as in the mBO equation. The spaces of these structures were first
found and used by Ionescu and Kenig \cite{IK} to remove some
logarithmic divergence. Now we state our main results:

\begin{theorem}\label{t11}
Fix $0<c_0<\infty$. Let $s\geq 1/2$ and $\delta\geq c_0$. Assume
$u_0\in H^s$ and $\norm{u_0}_{L^2}\ll 1$. Then

(a) Existence. There exists $T=T(\norm{u_0}_{H^{1/2}},c_0)>0$
independent of $\delta$ and a solution $u$ to the mFDF equation
\eqref{eq:mFDF} (or its focusing version) satisfying
\begin{equation}
u\in F^{s}(T)\subset C([-T,T]:H^{s}),
\end{equation}
where the function space $F^{s}(T)$ will be defined later (see
section 2).

(b) Uniqueness. The solution mapping $u_0\rightarrow u$ is the
unique extension of the mapping $H^\infty\rightarrow
C([-T,T]:H^\infty)$.

(c) Lipschitz continuity. For any $R>0$, the mapping $u_0\rightarrow
u$ is Lipschitz continuous from $\{u_0\in
H^{s}:\norm{u_0}_{H^{s}}<R, \norm{u_0}_{L^2}\ll 1 \}$ to
$C([-T,T]:H^{s})$ uniformly for all $\delta\geq c_0$.

(d) Persistence of regularities. If in addition $u_0 \in H^{s_1}$
for some $s_1>s$, then the solution $u$ belongs to $H^{s_1}$
uniformly for all $\delta\geq c_0$.
\end{theorem}

From the a-priori bound \eqref{eq:H1halfpriori} and iterating
Theorem \ref{t11}, we obtain the following corollary.

\begin{corollary}
Fix $0<c_0<\infty$. The Cauchy problem for Eq. \eqref{eq:mFDF} (or
its focusing version) is uniformly globally wellposed if $\phi$
belongs to $H^s$ for $s\geq 1/2$ with $\norm{\phi}_{L^2}$
sufficiently small for all $\delta\geq c_0$.
\end{corollary}

\begin{remark}
Our methods also work for the complex-valued mFDF equation
\eqref{eq:mFDF}. We can obtain local wellposedness but with some
weaker uniqueness. See \cite{Guo}.

For the other mFDF equation \eqref{eq:mFDF2}, it is easy to see from
\eqref{eq:trans} that local wellposedness also holds. However, we
can not obtain uniform local (global) wellposedness for $0<\delta
\les 1$. This is the reason why we can not prove the limit behavior
in $C([0,T]:H^{1/2})$ as $\delta\rightarrow 0$ for Eq.
\eqref{eq:mFDF2} which we conjecture holds. Our results are sharp in
the following sense.
\end{remark}

\begin{theorem}\label{illposed}
Assume $\delta\ges 1$. If $s<1/2$, then the solution map of Eq.
\eqref{eq:mFDF} is not $C^3$ in $H^s$.
\end{theorem}

In the second part we study the limit behavior as $\delta\rightarrow
\infty$ for Eq. \eqref{eq:mFDF}. It is natural to conjecture that
the solution of Eq. \eqref{eq:mFDF} converges to that of
\eqref{eq:mBO} as $\delta\rightarrow \infty$. Indeed, denote by
$S_T^\delta$, $S_T$ the solution map of Eq. \eqref{eq:mFDF}, Eq.
\eqref{eq:mBO} in \cite{KenigT, Guo} and we proved the following
\begin{theorem}\label{limit}
Let $s\geq 1/2$. Assume $u_0 \in H^s(\R)$ with $\norm{u_0}_{L^2}\ll
1$. For any $T>0$, then
\begin{equation}\label{eq:limitthm}
\lim_{\delta\rightarrow
\infty}\norm{S_T^\delta(u_0)-S_T(u_0)}_{C([0,T],H^s)}=0.
\end{equation}
\end{theorem}

\begin{remark}
We are only concerned with the limit in the same regularity space.
There seems no convergence rate. This can be seen from the linear
solution,
\begin{eqnarray*}
\norm{\ft^{-1}e^{it[\coth(2\pi \delta \xi)-\frac{1}{2\pi \delta
\xi}]\xi^2}\ft
u_0-e^{t\Hl\partial_{x}^2}u_0}_{C([0,T],H^s)}\rightarrow 0, \quad
\mbox{as }\delta\rightarrow \infty,
\end{eqnarray*}
but without any convergence rate. If the initial data has higher
regularity, then there is a convergence rate. For example, we prove
that
\begin{eqnarray*}
\norm{u_\delta-v}_{C([0,T], H^{1/2})}\les
\norm{\phi_1-\phi_2}_{H^{1/2}}+\frac{1}{\delta}C(T,\norm{\phi_1}_{H^{3/2}},\norm{\phi_2}_{H^{1/2}}).
\end{eqnarray*}
For the limit behavior for the other form Eq. \eqref{eq:mFDF2} as
$\delta\rightarrow 0$, we can't prove the same results. One can
obtain the similar results as in \cite{ABFS} using the energy
methods.
\end{remark}

In proving Theorem \ref{limit} we will adopt the same ideas as we
did for the KdV-Burger equations \cite{GW}. Considering the
difference equation, we first treat the difference term
$(\G_\delta-\Hl)\partial_x^2 u$ as nonlinear term, then use the
uniform global well-posedness.

The rest of the paper is organized as following. In Section 2 we
present some notations and Banach function spaces. Some properties
of the space are given in Section 3. In Section 4 we prove symmetric
estimates that will be used to prove trilinear estimates in Section
5. Theorem \ref{t11}, \ref{illposed}, \ref{limit} are proved in
Section 6, 7, 8, respectively.

\section{Notation and Definitions}

Throughout this paper, we fix $0<c_0<\infty$. For $x, y>0$, $x\les
y$ means that there exists $C> 0$ that may depend on $c_0$ such that
$x\leq C y$. By $x\sim y$ we mean $x\les y$ and $y\les x$.
Similarly, we use $x\ges y$, $x\ll y$ and $x\gg y$. For $f\in \Sch'$
we denote by $\widehat{f}$ or $\ft (f)$ the Fourier transform of $f$
for both spatial and time variables,
\begin{eqnarray*}
\widehat{f}(\xi, \tau)=\int_{\R^2}e^{-ix \xi}e^{-it \tau}f(x,t)dxdt.
\end{eqnarray*}
We denote  by $\ft_x \ (\ft_t)$ the Fourier transform on spatial
variable (time variable). If there is no confusion, we still write
$\ft=\ft_x$. Let $\mathbb{Z}$ and $\mathbb{N}$ be the sets of
integers and natural numbers, respectively. Let $\Z_+=\Z\cap[0,
\infty)$. For $k\in \Z_+$ let ${I}_k=\{\xi: |\xi|\in [2^{k-1},
2^{k+1}]\}$ if $k\geq 1$ and $I_0=[-2,2]$.

 Let $\eta_0: \R\rightarrow [0, 1]$ denote an even
smooth function supported in $[-8/5, 8/5]$ and equal to $1$ in
$[-5/4, 5/4]$. For $k\in \Z$ let
$\chi_k(\xi)=\eta_0(\xi/2^k)-\eta_0(\xi/2^{k-1})$, $\chi_k$
supported in $\{\xi: |\xi|\in[(5/8)\cdot 2^k, (8/5)\cdot 2^k]\}$,
and
\[\chi_{[k_1,k_2]}=\sum_{k=k_1}^{k_2}\chi_k \mbox{ for any } k_1\leq k_2\in \Z.\]
For simplicity of notation, let $\eta_k=\chi_k$ if $k\geq 1$ and
$\eta_k\equiv 0$ if $k\leq -1$. Also, for $k_1\leq k_2\in \Z$ let
\[\eta_{[k_1,k_2]}=\sum_{k=k_1}^{k_2}\eta_k \mbox{ and }\eta_{\leq k_2}=\sum_{k=-\infty}^{k_2}\eta_{k}.\]
Roughly speaking, $\{\chi_k\}_{k\in \mathbb{Z}}$ is the homogeneous
decomposition function sequence and $\{\eta_k\}_{k\in \mathbb{Z}_+}$
is the non-homogeneous decomposition function sequence to the
frequency space. For $k\in \Z_+$ let $P_k$ denote the operator on
$L^2(\R)$ defined by
\[
\widehat{P_ku}(\xi)=\eta_k(\xi)\widehat{u}(\xi).
\]
By a slight abuse of notation we also define the operator $P_k$ on
$L^2(\R\times \R)$ by formula $\ft(P_ku)(\xi,
\tau)=\eta_k(\xi)\ft(u)(\xi, \tau)$. For $l\in \Z$ let
\[
P_{\leq l}=\sum_{k\leq l}P_k, \quad P_{\geq l}=\sum_{k\geq l}P_k.
\]

Let $a_1, a_2, a_3, a_4\in \R$. It will be convenient to define the
quantities $a_{max}\geq a_{sub}\geq a_{thd}\geq a_{min}$ to be the
maximum, sub-maximum, third-maximum, and minimum of $a_1,a_2,a_3,
a_4$ respectively. We also denote $\sub(a_1,a_2,a_3,a_4)=a_{sub}$
and $\thd(a_1,a_2,a_3,a_4)=a_{thd}$. Usually we use
$k_1,k_2,k_3,k_4$ and $j_1,j_2,j_3,j_4$ to denote integers,
$N_i=2^{k_i}$ and $L_i=2^{j_i}$ for $i=1,2,3,4$ to denote dyadic
numbers.

For $\xi\in \R$ let
\begin{equation}\label{eq:dr}
\dr(\xi)=[\coth(2\pi \delta \xi)-\frac{1}{2\pi \delta \xi}]\xi^2
\end{equation}
be the dispersion relation associated to Eq. \eqref{eq:mFDF}. The
elementary properties of the function $\dr(\xi)$ are given in Lemma
\ref{elemes}. For $\phi \in L^2(\R)$ let $\U(t)\phi\in C(\R:L^2)$
denote the solution of the free finite-depth-fluid evolution given
by
\begin{equation}
\ft_x[\U(t)\phi](\xi,t)=e^{it\dr(\xi)}\widehat{\phi}(\xi),
\end{equation}
where $\dr(\xi)$ is defined in \eqref{eq:dr}. For $k,j \in \Z_+$ let
$D_{k,j}=\{(\xi, \tau)\in \R \times \R: \xi \in I_k,
\tau-\dr(\xi)\in {I}_j\}$. We define first the Banach spaces
$X_k=X_k(\R^2)$. For $k\in \Z_+$ we define
\begin{eqnarray}\label{eq:Xk}
X_k&=&\{f\in L^2(\R^2): f(\xi,\tau) \mbox{ is supported in }
I_k\times\R
\mbox{ and }\nonumber\\
&& \norm{f}_{X_k}:=\sum_{j=0}^\infty
2^{j/2}\beta_{k,j}\norm{\eta_j(\tau-\dr(\xi))\cdot
f(\xi,\tau)}_{L^2_{\xi,\tau}}<\infty\},
\end{eqnarray}
where
\begin{equation}\label{eq:betakj}
\beta_{k,j}=1+2^{2(j-2k)/5}.
\end{equation}
The precise choice of the coefficients $\beta_{k,j}$ is important in
order for all the trilinear estimates to hold. This factor is
particularly important in controlling the high-low interaction.

The spaces $X_k$ are not sufficient for our purpose, due to various
logarithmic divergences involving the modulation variable. Fix $M>1$
to be a large integer which is dependent on $c_0$.  For $k\geq M$ we
also define the Banach spaces $Y_k=Y_k(\R^2)$. For $k\geq M$ we
define
\begin{eqnarray}\label{eq:Yk}
Y_k&=&\{f\in L^2(\R^2): f(\xi,\tau) \mbox{ is supported in }
\bigcup_{j=0}^{k-1}D_{k,j} \mbox{ and } \nonumber\\&&
\norm{f}_{Y_k}:=2^{-k/2}\norm{\ft^{-1}[(\tau-\dr(\xi)+i)f(\xi,\tau)]}_{L_x^1L_t^2}<\infty\}.
\end{eqnarray}
Then for $k\in \Z_+$ we define
\begin{equation}\label{eq:Zk}
Z_k:=X_k \mbox{ if } k\leq M-1 \mbox{ and } Z_k:=X_k+Y_k \mbox{ if }
k\geq M.
\end{equation}
The spaces $Z_k$ are our basic Banach spaces. For $s\geq 0$ we
define the Banach spaces $F^{s}=F^{s}(\R\times\R)$:
\begin{eqnarray}\label{eq:Fs}
F^{s}=\{u\in \Sch'(\R\times \R):
\norm{u}_{F^{s}}^2=\sum_{k=0}^{\infty}2^{2sk}\norm{\eta_k(\xi)\ft
(u)}_{Z_k}^2<\infty\},
\end{eqnarray}
and $N^{s}=N^{s}(\R\times\R)$ which is used to measure the nonlinear
term and can be viewed as an analogue of $X^{s,b-1}$
\begin{eqnarray}\label{eq:Ns}
N^{s}&=&\{u\in \Sch'(\R\times \R):\nonumber\\
&&\norm{u}_{N^{s}}^2=\sum_{k=0}^{\infty}2^{2sk}\norm{\eta_k(\xi)(\tau-\dr(\xi)+i)^{-1}\ft
(u)}_{Z_k}^2<\infty\}.
\end{eqnarray}
We also define $F^s(T)$ and $N^s(T)$ to be the spaces that $F^s$ and
$N^s$ restricted to the time interval $[-T,T]$, respectively.

These $l^1$-type $X^{s,b}$ structures $X_k$ were first introduced
and used in \cite{Tataru, IK, IKT}. It is also useful in the study
of uniform global wellposedness and inviscid limit for the nonlinear
dispersive equation with dissipative term \cite{GW}. The combination
of $X^{s,b}$ structure and smoothing effect $Z_k$ were first used by
Ionescu and Kenig \cite{IK}.

\section{Properties of the spaces $Z_k$}

In this section we devote to study the properties of the spaces
$Z_k$. We start with some elementary estimates on the dispersion
relation $\dr(\xi)$ some of which were also proved in \cite{HWang}.
\begin{lemma}\label{elemes}
If $\delta>0$, then
\begin{eqnarray}
\left \{
\begin{array}{l}
|\dr(\xi)|\sim |\xi|^2,\ |\dr'(\xi)|\sim |\xi|,\ |\dr''(\xi)|\sim 1; \quad if\ |\xi|\ges 1/\delta.\\
|\dr(\xi)|\sim \delta|\xi|^3,\ |\dr'(\xi)|\sim \delta|\xi|^2,\
|\dr''(\xi)|\sim \delta |\xi|; \quad if\ |\xi|\les 1/\delta.
\end{array}
\right.
\end{eqnarray}
\end{lemma}
\begin{proof}
Since $\dr(\cdot)$ is odd, we may assume $\xi>0$. Let
$h(\xi)=[\coth(\xi)-\frac{1}{\xi}]\xi^2$, then we see that
$\dr(\xi)=\rev{4\pi^2 \delta^2}h(2\pi \delta \xi)$. Using Taylor's
expansion, we get
\[h(\xi)=\frac{\sum_{k=0}^\infty (\rev{(2k)!}-\rev{(2k+1)!})2\xi^{2k+1}}{\xi(e^\xi-e^{-\xi})}\xi^2 >0, \quad if\ \xi>0.\]
From $\lim\limits_{\xi\rightarrow 0+}\frac{h(\xi)}{\xi^3}>0$ and
$\lim\limits_{\xi \rightarrow +\infty}\frac{h(\xi)}{\xi^2}>0$, we
get that $|h(\xi)|\sim |\xi|^2$ if $|\xi|\ges 1$ and $|h(\xi)|\sim
|\xi|^3$ if $|\xi|\les 1$. Direct computations show that
\begin{eqnarray*}
&&h'(\xi)=\brk{\rev{\xi^2}-\frac{4}{(e^\xi-e^{-\xi})^2}}\xi^2+[\coth(\xi)-\frac{1}{\xi}]2\xi;\\
&&h''(\xi)=\frac{8\xi^2(e^\xi+e^{-\xi})-16\xi(e^\xi-e^{-\xi})+2(e^\xi+e^{-\xi})(e^\xi-e^{-\xi})^2}{(e^\xi-e^{-\xi})^3}.
\end{eqnarray*}
Using Taylor's expansion, we easily see that if $\xi>0$ then
\[h''(\xi)>\frac{16\xi^2(e^\xi+e^{-\xi})-16\xi(e^\xi-e^{-\xi})}{(e^\xi-e^{-\xi})^3}>0.\]
From $\lim\limits_{\xi\rightarrow 0+}\frac{h''(\xi)}{\xi}>0$ and
$\lim\limits_{\xi \rightarrow +\infty}h(\xi)>0$, we get that
$|h''(\xi)|\sim 1$ if $|\xi|\ges 1$ and $|h''(\xi)|\sim |\xi|$ if $
|\xi|\les 1$. Similarly, we get $|h'(\xi)|\sim |\xi|$ if $|\xi|\ges
1$ and $|h'(\xi)|\sim |\xi|^2$ if $|\xi|\les 1$. Therefore, we
complete the proof of the lemma.
\end{proof}

For $\xi_1,\xi_2,\xi_3\in \R$, let
\begin{eqnarray}\label{eq:reso}
\Omega(\xi_1,\xi_2,\xi_3)=\dr(\xi_1)+\dr(\xi_2)+\dr(\xi_3)-\dr(\xi_1+\xi_2+\xi_3).
\end{eqnarray}
This is the resonance function which plays crucial rule in the
trilinear estimate. See \cite{Taokz} for more perspective
discussion. We prove an estimate on the resonance in the following
lemma.
\begin{lemma}\label{esresonance}
Let $\delta\geq c_0$. Assume $|\xi_1|\leq |\xi_2|\leq |\xi_3|$,
$|\xi_1|\ll |\xi_3|$, $|\xi_1+\xi_2+\xi_3|\sim |\xi_3|$ and
$|\xi_3|\gg 1$. Then we have
\begin{eqnarray}\label{eq:esreso}
|\Omega(\xi_1,\xi_2,\xi_3)|\sim |\xi_1+\xi_2|\cdot |\xi_3|.
\end{eqnarray}
\end{lemma}
\begin{proof}
We consider first the case that $|\xi_2|\ll |\xi_3|$. From the mean
value formula we see that
\[|\dr(\xi_3)-\dr(\xi_1+\xi_2+\xi_3)|\sim |\xi_3|\cdot|\xi_1+\xi_2|, \ |\dr(\xi_1)+\dr(\xi_2)|\ll |\xi_3|\cdot|\xi_1+\xi_2|,\]
which immediately gives \eqref{eq:esreso} in this case.

We consider now the case $|\xi_2|\ges |\xi_3|$. Then it suffices to
show that $|\Omega(\xi_1,\xi_2,\xi_3)|\sim  |\xi_3|^2$. We get from
the definition that
\begin{eqnarray*}
&&\dr(\xi_2)+\dr(\xi_3)-\dr(\xi_2+\xi_3)\\
&=&\coth(2\pi \delta \xi_2){\xi_2^2}+\coth(2\pi \delta
\xi_3){\xi_3^2}-\coth(2\pi \delta (\xi_2+\xi_3)){(\xi_2+\xi_3)^2}\\
&=&[\coth(2\pi \delta \xi_2)-\coth(2\pi \delta
(\xi_2+\xi_3))]\xi_2^2-\coth(2\pi \delta
(\xi_2+\xi_3))2\xi_2\xi_3\\
&&+[\coth(2\pi \delta \xi_3)-\coth(2\pi \delta
(\xi_2+\xi_3))]\xi_3^2=:I+II+III.
\end{eqnarray*}
It is easy to see that $|I|,\ |III|\ll |\xi|^2$ and $|II|\sim
|\xi_3|^2$, then \eqref{eq:esreso} follows from the fact that
$|\dr(\xi_1+\xi_2+\xi_3)-\dr(\xi_2+\xi_3)|\ll |\xi_3|^2$.
\end{proof}

From the definitions we see that if $k\in \Z_+$ and $f_k \in Z_k$
then $f_k$ can be written in the form
\begin{eqnarray}
\left \{
\begin{array}{l}
f_k=\sum_{j=0}^{\infty}f_{k,j}+g_k;\\
\sum_{j=0}^{\infty}2^{j/2}\beta_{k,j}\norm{f_{k,j}}_{L^2}+\norm{g_k}_{Y_k}\leq
2\norm{f_k}_{Z_k},
\end{array}
\right.
\end{eqnarray}
such that $f_{k,j}$ is supported in $D_{k,j}$ and $g_k$ is supported
in $\cup_{j=0}^{k-1}D_{k,j}$ (if $k\leq M-1$ then $g_k\equiv 0$). In
analogy with Lemma 4.1 in \cite{IK} we have the following
\begin{lemma}\label{basicproperties} (a) If $m, m':\R\rightarrow \C$,
$k\in \Z_+$, and $f_k\in Z_k$ then
\begin{eqnarray}
\left \{
\begin{array}{l}
\norm{m(\xi)f_k(\xi,\tau)}_{Z_k}\leq C\norm{\ft^{-1}(m)}_{L^1(\R)}\norm{f_k}_{Z_k};\\
\norm{m'(\tau)f_k(\xi,\tau)}_{Z_k}\leq
C\norm{m'}_{L^\infty(\R)}\norm{f_k}_{Z_k}.
\end{array}
\right.
\end{eqnarray}

(b) If $\delta\geq c_0$, $k\in \Z_+$, $j\geq 0$, and $f_k\in Z_k$
then
\begin{equation}
\norm{\eta_j(\tau-\dr(\xi))f_k(\xi,\tau)}_{X_k}\les
\norm{f_k}_{Z_k}.
\end{equation}

(c) If $\delta\geq c_0$, $k\geq 1$, $j\in [0,k]$, and $f_k$ is
supported in $I_k\times \R$ then
\begin{equation}
\norm{\ft^{-1}[\eta_{\leq
j}(\tau-\dr(\xi))f_k(\xi,\tau)]}_{L_x^1L_t^2}\les
\norm{\ft^{-1}(f_k)}_{L_x^1L_t^2}.
\end{equation}
\end{lemma}
\begin{proof}
It is easy to see that part (a) follows directly from Plancherel
theorem and the definitions.

For part (b), we may assume $k\geq M$, $f_k=g_k\in Y_k$, and $j\leq
k$. From the definition we see that if $g_k \in Y_k$ then $g_k$ can
be written in the form
\begin{eqnarray}\label{eq:gkform}
\left\{\begin{array}{l}
g_k(\xi,\tau)=2^{k/2}\chi_{[k-1,k+1]}(\xi)(\tau-\dr(\xi)+i)^{-1}\eta_{\leq
k}(\tau-\dr(\xi))\ft_x(h)(\xi,\tau);\\
\norm{g_k}_{Y_k}=C\norm{h}_{L_x^1L_\tau^2}.
\end{array}\right.
\end{eqnarray}
Thus from Plancherel's equality we get
\begin{eqnarray*}
&&\norm{\eta_j(\tau-\dr(\xi))g_k(\xi,\tau)}_{X_k}\\
&\les& \sup_{0\leq
j\leq k}
2^{-j/2}2^{k/2}\norm{\chi_k(\xi)\eta_j(\tau-\dr(\xi))\ft_x(h)(\xi,\tau)}_{L_{\xi,\tau}^2}\\
&\les&\sup_{0\leq j\leq k}
2^{-j/2}2^{k/2}\norm{\ft_x^{-1}[\chi_k(\xi)\eta_j(\tau-\dr(\xi))]}_{L_\tau^\infty
L_x^2}\norm{h}_{L_x^1L_\tau^2}.
\end{eqnarray*}
On the other hand, by changing of variable $\mu=\dr(\xi)$ we get
from Lemma \ref{elemes} and $\delta\geq c_0$ that
\[\norm{\chi_k(\xi)\eta_j(\tau-\dr(\xi))}_{L_\tau^\infty
L_\xi^2}\les 2^{j-k},\] which completes the proof of part (b).

For part (c), from Plancherel's equality, it suffices to prove that
\begin{eqnarray}\label{eq:l33c}
\normo{\int_\R e^{ix\xi} \chi_k(\xi)\eta_{\leq
j}(\tau-\dr(\xi))d\xi}_{L_x^1L_\tau^\infty}\leq C.
\end{eqnarray}
We may assume $k\geq M$ in proving \eqref{eq:l33c}. By the change of
variable $\tau-\dr(\xi)=\alpha$, integration by parts and Lemma
\ref{elemes} we obtain that
\begin{eqnarray*}
\aabs{\int_\R e^{ix\xi} \chi_k(\xi)\eta_{\leq
j}(\tau-\dr(\xi))d\xi}\les \  \frac{2^{j-k}}{1+(2^{j-k}x)^2}
\end{eqnarray*}
which suffices to prove \eqref{eq:l33c}.
\end{proof}

We study now the embedding properties of the spaces $Z_k$ which is
important in the trilinear estimates.
\begin{lemma}\label{embedgen} Assume $\delta\geq c_0$. Let
$k\in \Z_+$, $s\in \R$ and $I\subset \R$ be an interval. Let $Y$ be
$L_x^pL_{t\in I}^q$ or $L_{t\in I}^qL_x^p$ for some $1\leq p,\ q\leq
\infty$ with the property that
\[\norm{\U(t)f}_Y\les 2^{ks}\norm{f}_{L^2(\R)}\]
for all $f\in L^2(\R)$ with $\widehat{f}$ supported in ${I}_k$. Then
we have that if $f_k\in Z_k$
\begin{equation}
\norm{\ft^{-1}(f_k)}_{Y}\les 2^{ks}\norm{f_k}_{Z_k}.
\end{equation}
\end{lemma}
\begin{proof}
We assume first that $f_k=f_{k,j}$ with
$\norm{f_k}_{X_k}=2^{j/2}\beta_{k,j}\norm{f_{k,j}}_{L^2}$ and
$f_{k,j}$ is supported in $D_{k,j}$ for some $j\geq 0$. Then we have
\begin{eqnarray*}
\ft^{-1}(f_{k})(x,t)&=&\int f_{k,j}(\xi,\tau) e^{ix\xi}e^{it\tau}d\xi d\tau\\
&=&\int_{{I}_j} e^{it\tau} \int f_{k,j}(\xi,\tau+\dr(\xi))
e^{ix\xi}e^{it\dr(\xi)}d\xi d\tau.
\end{eqnarray*}
From the hypothesis on $Y$, we obtain
\begin{eqnarray*}
\norm{\ft^{-1}(f_{k})(x,t)}_Y &\les&  \int \eta_j(\tau)
\normo{e^{it\tau} \int f_{k,j}(\xi,\tau+\dr(\xi))
e^{ix\xi}e^{it\dr(\xi)}d\xi}_{Y}d\tau\\
&\les& 2^{ks}2^{j/2}\norm{f_{k,j}}_{L^2},
\end{eqnarray*}
which completes the proof in this case.

We assume now that $k\geq M$ and $f_k=g_k\in Y_k$. From definitions
and \eqref{eq:gkform}, it suffices to prove that if
\begin{eqnarray}\label{eq:l34f}
f(\xi,\tau)=2^{k/2}\chi_{[k-1,k+1]}(\xi)(\tau-\dr(\xi)+i)^{-1}\eta_{\leq
k}(\tau-\dr(\xi))\cdot h(\tau)
\end{eqnarray}
then
\begin{eqnarray}\label{eq:l34Yk}
\normo{\int_{\R^2}f(\xi,\tau)e^{ix\xi}e^{it\tau}d\xi d\tau}_{Y}\les
2^{ks}\norm{h}_{L^2}.
\end{eqnarray}
Since $k\geq 100$ and $|\xi|\in [2^{k-2},2^{k+2}]$, we may assume
that the function $h$ in \eqref{eq:l34f} is supported in the set
$\{t: |\tau|\in [2^{2k-C},2^{2k+C}]\}$. Let $h_+=h\cdot
1_{[0,\infty)}$, $h_-=h\cdot 1_{(-\infty,0]}$, and define the
corresponding functions $f_+$ and $f_-$ as in \eqref{eq:l34f}. By
symmetry, it suffices to prove the bounds \eqref{eq:l34Yk} for the
function $f_+$, which is supported in the set $\{(\xi,\tau):\xi \in
[2^{k-2},2^{k+2}],\ \tau\in [2^{2k-C},2^{2k+C}]\}$. From the fact
that $\dr(\xi)$ is strictly increasing, we have an inverse function
$\varphi(\alpha)=\dr^{-1}(\alpha)$. It follows from Lemma
\ref{elemes} that $\varphi(\tau)\sim 2^k$ for $\tau\in
[2^{2k-C},2^{2k+C}]$. Thus $f_+(\xi,\tau)\equiv 0$ unless
$|\varphi(\tau)-\xi|\leq C$. Let
\begin{eqnarray}\label{eq:l34fprime}
f_+'(\xi,\tau)&=&2^{k/2}\chi_{[k-1,k+1]}(\varphi(\tau))(\varphi(\tau)^2-\xi^2+(\varphi(\tau)-\xi)^2+i\varphi(\tau)2^{-k})^{-1}\nonumber\\
& &\cdot \eta_0(\varphi(\tau)-\xi)\cdot h_+(\tau)
\end{eqnarray}
Since for $\varphi(\tau)\sim 2^k$, $\tau\in [2^{2k-C},2^{2k+C}]$ and
$|\varphi(\tau)-\xi|\leq C$ we have
\begin{eqnarray*}
|\tau-\dr(\xi)-\varphi(\tau)^2+\xi^2|&=&|\dr(\varphi(\tau))-\dr(\xi)-\varphi(\tau)^2+\xi^2|\leq
C,
\end{eqnarray*}
it is easy to see that
\[\norm{f_+-f_+'}_{X_k}\leq C\norm{h_+}_{L^2}.\]
Thus using the estimate for $f\in X_k$, we get
$\norm{\ft^{-1}(f_+-f_+')}_{Y}\leq C2^{ks}\norm{h_+}_{L^2}$. It
remains to estimate $\norm{\ft^{-1}(f_+')}_{Y}$. We make the change
of variables $\xi=\varphi(\tau)-\mu$, then
\begin{eqnarray}
\ft^{-1}(f_+')(x,t)&=&2^{k/2}\int_{\R} h_+(\tau)
(\varphi(\tau))^{-1}\chi_{[k-1,k+1]}(\varphi(\tau))e^{it\tau}e^{ix\varphi(\tau)}d\tau \nonumber\\
&&\cdot \int_\R \eta_0(\mu)(\mu+i/2^{k+1})^{-1}e^{ix\mu}d\mu.
\end{eqnarray}
The second integral is bounded by $C$. We make the change of
variable $\xi=\varphi(\tau)$ in the first integral, then by the
hypothesis of $Y$ we get
\[\norm{\ft^{-1}(f_+')}_{Y}\les 2^{ks} 2^{k/2}\norm{h_+(\dr(\xi))\chi_k(\xi)}_{L^2}\les 2^{ks}\norm{h_+}_{L^2}.\]
Therefore, we complete the proof of the lemma.
\end{proof}

In order to obtain the more specific embedding properties of the
spaces $Z_k$, we need the estimates for the free finite-depth-fluid
equation. We prove the Strichartz estimates, smoothing effects, and
maximal function estimates for the free solutions in the following
lemma.

\begin{lemma}\label{lineares}
Assume $\delta\geq c_0$. Let $k\in \Z_+$ and $I\subset \R$ be an
interval with $|I|\les 1$. Then for all $\phi \in L^2(\R)$ with
$\widehat{\phi}$ supported in $I_k$,

(a) Strichartz estimates: if $k\geq 1$ then
\begin{equation}\label{eq:l33a}
\norm{\U(t)\phi}_{L_t^qL_x^r}\les \norm{\phi}_{L^2(\R)},
\end{equation}
where $(q,r)$ is admissible, namely $2\leq q,r\leq \infty$ and
$2/q=1/2-1/r$.

(b) Smoothing effect: if $k\geq 1$ then
\begin{equation}
\norm{\U(t)\phi}_{L_x^\infty L_t^2}\les
2^{-k/2}\norm{\phi}_{L^2(\R)}.
\end{equation}

(c) Maximal function estimate: if $k\geq 0$ then
\begin{eqnarray}
\norm{\U(t)\phi}_{L_x^2 L_{t\in I}^\infty}&\les&
2^{k/2}\norm{\phi}_{L^2(\R)},\\
\norm{\U(t)\phi}_{L_x^4 L_{t}^\infty}&\les&
2^{k/4}\norm{\phi}_{L^2(\R)}.
\end{eqnarray}
\end{lemma}

\begin{proof}
For part (a), we use the results in \cite{GPW} and Lemma
\ref{elemes}. We easily see that the Strichartz estimates
\eqref{eq:l33a} holds if $(q,r)$ is admissible pairs as for
Schr\"odinger equation.

Part (b) and the second inequality in part (c) follow from Lemma
\ref{elemes} and the results in \cite{KPV2}. They were also proved
in \cite{HWang}. The first inequality in part (c) follows from
slightly modified argument in the proof of theorem 3.1 in
\cite{KPV3}. We omit the details.
\end{proof}

From Lemma \ref{embedgen} and Lemma \ref{lineares}, we immediately
get the following

\begin{lemma}\label{embedding}
Assume $\delta\geq c_0$. Let $k \in \Z_+$ and $I\subset \R$ be an
interval with $|I|\les 1$. Assume $(q,r)$ is admissible and $f_k \in
Z_k$. Then

(a) If $k\geq 1$, then
\[\norm{\ft^{-1}(f_k)}_{L_x^\infty L_t^2}\les
2^{-k/2}\norm{f_k}_{Z_k},\ \norm{\ft^{-1}(f_k)}_{L_t^q L_x^r}\les
\norm{f_k}_{Z_k}.\]

(b) For all $k\in \Z_+$, \[\norm{\ft^{-1}(f_k)}_{L_x^4
L_t^\infty}\les 2^{k/4}\norm{f_k}_{Z_k},\
\norm{\ft^{-1}(f_k)}_{L_x^2L_{t\in I}^\infty}\les
2^{k/2}\norm{f_k}_{Z_k}.\] As a consequence, for any $s\geq 0$ we
have $F^{s}\subseteq C(\R; {H}^s)$ with $\norm{u}_{L_t^\infty
H^s}\les \norm{u}_{F^s}$.
\end{lemma}

\section{A Symmetric Estimate}
Following the standard fixed point argument, we will need to prove a
trilinear estimate. We start with a symmetric estimate for
nonnegative functions. For $\xi_1,\xi_2,\xi_3\in \R$ and
$\Omega:\R^3 \rightarrow \R$ as in \eqref{eq:reso}, and for
compactly supported nonnegative functions $f,g,h,u \in L^2(\R\times
\R)$ let
\begin{eqnarray*}
J(f,g,h,u)&=&\int_{\R^6}f(\xi_1,\mu_1)g(\xi_2,\mu_2)h(\xi_3,\mu_3)\\
&&u(\xi_1+\xi_2+\xi_3,\mu_1+\mu_2+\mu_3+\Omega(\xi_1,\xi_2,\xi_3))d\xi_1d\xi_2d\xi_3d\mu_1d\mu_2d\mu_3.
\end{eqnarray*}

\begin{lemma}\label{l41}
Assume $\delta \geq c_0$. Assume $k_i, j_i \in \Z_+$ and
$N_i=2^{k_i}, L_i=2^{j_i}$  and $f_{k_i,j_i}\in L^2(\R\times \R)$
are nonnegative functions supported in $I_{k_i}\times {I}_{j_i}, \
i=1,2,3,4$. For simplicity we write
$J=|J(f_{k_1,j_1},f_{k_2,j_2},f_{k_3,j_3},f_{k_4,j_4})|$.

(a) For any $k_i, j_i\in \Z_+, i=1,2,3,4$,
\begin{equation}
J\les \ 2^{(j_{min}+j_{thd})/2}2^{(k_{min}+k_{thd})/2}
\prod_{i=1}^4\norm{f_{k_i,j_i}}_{L^2}.
\end{equation}

(b) If $N_{thd}\ll N_{sub}$, $N_{max}\gg 1$, and $(k_i, j_i)\neq
(k_{thd},j_{max})$ for $i=1,2,3,4$,
\begin{equation}\label{eq:l41rb}
J\les \
2^{(j_1+j_2+j_3+j_4)/2}2^{-j_{max}/2}2^{-k_{max}/2}2^{k_{min}/2}
\prod_{i=1}^4\norm{f_{k_i,j_i}}_{L^2};
\end{equation}
if $N_{thd}\ll N_{sub}$, $N_{max}\gg 1$, and $(k_i,
j_i)=(k_{thd},j_{max})$ for some $i\in \{1,2,3,4\}$,
\begin{equation}
J\les \
2^{(j_1+j_2+j_3+j_4)/2}2^{-j_{max}/2}2^{-k_{max}/2}2^{k_{thd}/2}
\prod_{i=1}^4\norm{f_{k_i,j_i}}_{L^2}.
\end{equation}

(c) For any $k_i, j_i\in \Z_+, i=1,2,3,4,$ with $N_{max}\gg 1$,
\begin{equation}
J\les \ 2^{(j_1+j_2+j_3+j_4)/2}2^{-j_{max}/2}
\prod_{i=1}^4\norm{f_{k_i,j_i}}_{L^2}.
\end{equation}

(d) If $N_{min}\ll N_{max}$ and $N_{max}\gg 1$, then
\begin{equation}
J\les \ 2^{(j_1+j_2+j_3+j_4)/2}2^{-k_{max}}
\prod_{i=1}^4\norm{f_{k_i,j_i}}_{L^2}.
\end{equation}

\end{lemma}
\begin{proof}
Let $A_{k_i}(\xi)=[\int_\R |f_{k_i,j_i}(\xi,\mu)|^2d\mu]^{1/2}$,
$i=1,2,3,4$. Using the Cauchy-Schwarz inequality and the support
properties of the functions $f_{k_i,j_i}$,
\begin{eqnarray*}
&&|J(f_{k_1,j_1},f_{k_2,j_2},f_{k_3,j_3},f_{k_4,j_4})|\\
&\les&
2^{(j_{min}+j_{thd})/2}\int_{\R^3}A_{k_1}(\xi_1)A_{k_2}(\xi_1)A_{k_3}(\xi_1)A_{k_4}(\xi_1+\xi_2+\xi_3)d\xi_1d\xi_2d\xi_3\\
&\les&
2^{(k_{min}+k_{thd})/2}2^{(j_{min}+j_{thd})/2}\prod_{i=1}^4\norm{f_{k_i,j_i}}_{L^2},
\end{eqnarray*}
which is part (a), as desired.

For part (b), we observe that
$J(f_{k_1,j_1},f_{k_2,j_2},f_{k_3,j_3},f_{k_4,j_4})\equiv 0$ unless
\begin{equation}\label{eq:freeq}
N_{max}\sim N_{sub}.
\end{equation}
 Simple changes of variables in the integration and the fact
that the function $\omega$ is odd show that
\begin{eqnarray*}
|J(f,g,h,u)|=|J(g,f,h,u)|=|J(f,h,g,u)|=|J(\widetilde{f},g,u,h)|,
\end{eqnarray*}
where $\widetilde{f}(\xi,\mu)=f(-\xi,-\mu)$. Thus we may assume
$k_1\leq k_2\leq k_3\leq k_4$. We assume first that $j_2\neq
j_{max}$. Then we have several cases: if $j_4=j_{max}$, then we will
prove that if $g_i:\R\rightarrow \R_+$ are $L^2$ functions supported
in $I_{k_i}$, $i=1,2,3$, and $g: \R^2\rightarrow \R_+$ is an $L^2$
function supported in $I_{k_4}\times \widetilde{I}_{j_4}$, then
\begin{eqnarray}\label{eq:l41b}
&&\int_{\R^3}g_1(\xi_1)g_2(\xi_2)g_3(\xi_3)g(\xi_1+\xi_2+\xi_3,\Omega(\xi_1,\xi_2,\xi_3))d\xi_1d\xi_2d\xi_3\nonumber\\
&&\les \
2^{-k_{max}/2}2^{k_{min}/2}\norm{g_1}_{L^2}\norm{g_2}_{L^2}\norm{g_3}_{L^2}\norm{g}_{L^2}.
\end{eqnarray}
This suffices for \eqref{eq:l41rb}.

To prove \eqref{eq:l41b}, we first observe that since $N_{thd}\ll
N_{sub}$ then $|\xi_3+\xi_2|\sim |\xi_3|$. By change of variable
$\xi'_1=\xi_1$, $\xi'_2=\xi_2$, $\xi'_3=\xi_3+\xi_2$, we get that
the left side of \eqref{eq:l41b} is bounded by
\begin{eqnarray}\label{eq:l41b2}
&&\int_{|\xi_1|\sim 2^{k_1},|\xi_2|\sim 2^{k_2},|\xi_3|\sim
2^{k_3}}g_1(\xi_1)g_2(\xi_2)\nonumber\\
&&g_3(\xi_3-\xi_2)g(\xi_1+\xi_3,\Omega(\xi_1,\xi_2,\xi_3-\xi_2))d\xi_1d\xi_2d\xi_3.
\end{eqnarray}
Note that in the integration area we have
\begin{eqnarray*}
\big|\frac{\partial}{\partial_{\xi_2}}\left[\Omega(\xi_1,\xi_2,\xi_3-\xi_2)\right]\big|=|\dr'(\xi_2)-\dr'(\xi_3-\xi_2)|\sim
2^{k_3},
\end{eqnarray*}
where we use by Lemma \ref{elemes} that $\dr'(\xi)\sim |\xi|$ and
$N_2\ll N_3$. By change of variable
$\mu_2=\Omega(\xi_1,\xi_2,\xi_3-\xi_2)$, we get that
\eqref{eq:l41b2} is bounded by
\begin{eqnarray}
&&2^{-k_3/2}\int_{|\xi_1|\sim
2^{k_1}}g_1(\xi_1)\norm{g_2}_{L^2}\norm{g_3}_{L^2}\norm{g}_{L^2}d\xi_1 \nonumber\\
&&\les \
2^{-k_{max}/2}2^{k_{min}/2}\norm{g_1}_{L^2}\norm{g_2}_{L^2}\norm{g_3}_{L^2}\norm{g_1}_{L^2}\norm{g}_{L^2}.
\end{eqnarray}

If $j_3=j_{max}$, this case is identical to the case $j_4=j_{max}$
in view of \eqref{eq:freeq}. If $j_1=j_{max}$ it suffices to prove
that if $g_i:\R\rightarrow \R_+$ are $L^2$ functions supported in
$I_{k_i}$, $i=2,3,4$, and $g: \R^2\rightarrow \R_+$ is an $L^2$
function supported in $I_{k_1}\times \widetilde{I}_{j_1}$, then
\begin{eqnarray}\label{eq:l41b3}
&&\int_{\R^3}g_2(\xi_2)g_3(\xi_3)g_4(\xi_4)g(\xi_2+\xi_3+\xi_4,\Omega(\xi_2,\xi_3,\xi_4))d\xi_2d\xi_3d\xi_4\nonumber\\
&&\les \
2^{-k_{max}/2}2^{k_{min}/2}\norm{g_2}_{L^2}\norm{g_3}_{L^2}\norm{g_4}_{L^2}\norm{g}_{L^2}.
\end{eqnarray}

Indeed, by change of variables
$\xi'_2=\xi_2,\xi'_3=\xi_3,\xi'_4=\xi_2+\xi_3+\xi_4$ and noting that
in the area $|\xi'_2|\sim 2^{k_2},|\xi'_3|\sim 2^{k_3},|\xi'_4|\sim
2^{k_1}$,
\begin{eqnarray*}
\big|\frac{\partial}{\partial_{\xi'_2}}\left[\Omega(\xi'_2,\xi'_3,\xi'_4-\xi'_2-\xi'_3)\right]\big|=|\dr'(\xi'_2)-\dr'(\xi'_4-\xi'_2-\xi'_3)|\sim
2^{k_3},
\end{eqnarray*}
we get from Cauchy-Schwarz inequality that
\begin{eqnarray}
&&\int_{\R^3}g_2(\xi_2)g_3(\xi_3)g_4(\xi_4)g(\xi_2+\xi_3+\xi_4,\Omega(\xi_2,\xi_3,\xi_4))d\xi_2d\xi_3d\xi_4\nonumber\\
&\les& \int_{|\xi'_2|\sim 2^{k_2},|\xi'_3|\sim 2^{k_3}, |\xi'_4|\sim
2^{k_1}}g_2(\xi'_2)g_3(\xi'_3)\nonumber\\
&&\quad \cdot
g_4(\xi'_4-\xi'_2-\xi'_3)g(\xi'_4,\Omega(\xi'_2,\xi'_3,\xi'_4-\xi'_2-\xi'_3))d\xi'_2d\xi'_3d\xi'_4\nonumber\\
&\les&2^{-k_3/2}\int_{|\xi'_3|\sim 2^{k_3}, |\xi'_4|\sim
2^{k_1}}g_3(\xi'_3)\norm{g_2(\xi'_2)g_4(\xi'_4-\xi'_2-\xi'_3)}_{L_{\xi'_2}^2}\norm{g(\xi'_4,\cdot)}_{L_{\xi'_2}^2}d\xi'_3d\xi'_4\nonumber\\
&\les&2^{-k_{max}/2}2^{k_{min}/2}\norm{g_2}_{L^2}\norm{g_3}_{L^2}\norm{g_4}_{L^2}\norm{g}_{L^2}.
\end{eqnarray}

We assume now that $j_2=j_{max}$. The proof is identical to the case
$j_1=j_{max}$. We note that we actually prove that if $N_{thd}\ll
N_{sub}$ then
\begin{eqnarray}\label{eq:l41bb}
J\leq C
2^{(j_1+j_2+j_3+j_4)/2}2^{-j_{sub}/2}2^{-k_{max}/2}2^{k_{min}/2}\prod_{i=1}^4\norm{f_{k_i,j_i}}_{L^2}.
\end{eqnarray}
Therefore, we complete the proof for part (b).

For part (c), setting $f^{\sharp}_{k_i,j_i}=f_{k_i,j_i}(\xi,
\tau+\dr(\xi))$, $i=1,2,3$, then we get
\begin{eqnarray*}
&&|J(f_{k_1,j_1},f_{k_2,j_2},f_{k_3,j_3},f_{k_4,j_4})|\\
&=&|\int
f^\sharp_{k_1,j_1}*f^\sharp_{k_2,j_2}*f^\sharp_{k_3,j_3}\cdot
f^\sharp_{k_4,j_4}|\\
&\les&
\norm{f^\sharp_{k_1,j_1}*f^\sharp_{k_2,j_2}*f^\sharp_{k_3,j_3}}_{L^2}\norm{f^\sharp_{k_4,j_4}}_{L^2}\\
&\les&\norm{\ft^{-1}(f^\sharp_{k_1,j_1})}_{L_x^4L_t^\infty}\norm{\ft^{-1}(f^\sharp_{k_2,j_2})}_{L_x^4L_t^\infty}\norm{\ft^{-1}(f^\sharp_{k_3,j_3})}_{L_x^\infty
L_t^2}\norm{f_{k_4,j_4}|}_{L^2}.
\end{eqnarray*}
On the other hand,
\begin{eqnarray*}
\ft^{-1}(f^\sharp_{k_1,j_1})&=&\int_{\R^2}f_{k_i,j_i}(\xi,
\tau+\dr(\xi))e^{ix\xi}e^{it\tau}d\xi d\tau\\
&=&\int_{\R^2} f_{k_i,j_i}(\xi,
\tau)e^{ix\xi}e^{it\dr(\xi)}e^{it\tau}d\xi d\tau,
\end{eqnarray*}
then it follows from Lemma \ref{lineares} (c) that
\[\norm{\ft^{-1}(f^\sharp_{k_1,j_1})}_{L_x^4L_t^\infty}\les \
2^{j_1/2}2^{k_1/4}\norm{f_{k_1,j_1}}_{L^2}.\] Similarly we can bound
the other terms. Thus part (c) follows form the symmetry.

For part (d), we only need to consider the worst cases $\xi_1\cdot
\xi_2<0$ and $N_{thd}\ll N_{sub}$. Indeed in the other cases we get
from the fact $|\Omega(\xi_1,\xi_2,\xi_3)|\ges \ N_{thd}N_{sub}$
which implies that $L_{max}\ges \ N_{thd}N_{sub}$ by checking the
support properties. Thus (d) follows from (b) and (c) in these
cases. We assume now $\xi_1\cdot \xi_2<0$ and $N_{2}\ll N_{3}$. If
$j_4=j_{max}$, it suffices to prove that if $g_i$ is $L^2$
nonnegative functions supported in $I_{k_i}$, $i=1,2,3$, and $g$ is
a $L^2$ nonnegative function supported in $I_{k_4}\times {I}_{j_4}$,
then
\begin{eqnarray}\label{eq:l41d1}
&&\int_{\R^3\cap \{\xi_1\cdot
\xi_2<0\}}g_1(\xi_1)g_2(\xi_2)g_3(\xi_3)g(\xi_1+\xi_2+\xi_3,
\Omega(\xi_1,\xi_2,\xi_3))d\xi_1d\xi_2d\xi_3\nonumber\\
&\les&
2^{j_4/2}2^{-k_3}\norm{g_1}_{L^2}\norm{g_2}_{L^2}\norm{g_3}_{L^2}\norm{g}_{L^2}.
\end{eqnarray}
By localizing $|\xi_1+\xi_2|\sim 2^l$ for $l\in \Z$, we get that the
right-hand side of \eqref{eq:l41d1} is bounded by
\begin{eqnarray}\label{eq:l41d2}
\sum_{l}\int_{\R^3}\chi_{l}(\xi_1+\xi_2)g_1(\xi_1)g_2(\xi_2)g_3(\xi_3)g(\xi_1+\xi_2+\xi_3,
\Omega(\xi_1,\xi_2,\xi_3))d\xi_1d\xi_2d\xi_3.
\end{eqnarray}
From the support properties of the functions $g_i,\ g$ and Lemma
\ref{esresonance} that in the integration area
\[|\Omega(\xi_1,\xi_2,\xi_3)|\sim |(\xi_1+\xi_2)(\xi_1+\xi_3)|\sim 2^{l+k_3},\]
We get that
\begin{equation}\label{eq:l41dsi}
L_{max}\ges \ 2^{l+k_3}.
\end{equation}
By changing variable of integration $\xi_1'=\xi_1+\xi_2$,
$\xi_2'=\xi_2$, $\xi_3'=\xi_1+\xi_3$, we obtain that
\eqref{eq:l41d2} is bounded by
\begin{eqnarray}\label{eq:l41d3}
&&\sum_{l}\int_{|\xi_1'|\sim 2^l,|\xi_2'|\sim 2^{k_2},|\xi_3'|\sim 2^{k_3}}\chi_{l}(\xi_1')g_1(\xi_1'-\xi_2')g_2(\xi_2')g_3(\xi_2'+\xi_3'-\xi_1')\nonumber\\
&&\quad g(\xi_2'+\xi_3',
\Omega(\xi_1'-\xi_2',\xi_2',\xi_2'+\xi_3'-\xi_1'))d\xi_1'd\xi_2'd\xi_3'.
\end{eqnarray}
Since in the integration area
\begin{eqnarray}\label{eq:l41djo}
\big|\frac{\partial}{\partial_{\xi_1'}}[\Omega(\xi_1'-\xi_2',\xi_2',\xi_2'+\xi_3'-\xi_1')]\big|
=|\dr'(\xi_1'-\xi_2')-\dr'(\xi_2'+\xi_3'-\xi_1')|\sim 2^{k_3},
\end{eqnarray}
then we get from \eqref{eq:l41djo} that \eqref{eq:l41d3} is bounded
by
\begin{eqnarray}
&&\sum_{l}\int_{|\xi_1'|\sim 2^l}\chi_{l}(\xi_1')\norm{g_1}_{L^2}\norm{g_3}_{L^2}\nonumber\\
&&\quad \norm{g_2(\xi_2')g(\xi_2'+\xi_3',
\Omega(\xi_1'-\xi_2',\xi_2',\xi_2'+\xi_3'-\xi_1'))}_{L^2_{\xi_2',\xi_3'}}d\xi_1'\nonumber\\
&\les&\sum_{l}2^{l/2}2^{-k_3/2}\norm{g_1}_{L^2}\norm{g_2}_{L^2}\norm{g_3}_{L^2}\norm{g}_{L^2}\nonumber\\
&\les&
2^{j_{max}/2}2^{-k_3}\norm{g_1}_{L^2}\norm{g_2}_{L^2}\norm{g_3}_{L^2}\norm{g}_{L^2},
\end{eqnarray}
where we used \eqref{eq:l41dsi} in the last inequality.

From symmetry we know the case $j_3=j_{max}$ is identical to the
case $j_4=j_{max}$, and the case $j_1=j_{max}$ is identical to the
case $j_2=j_{max}$, thus it reduces to prove the case $j_2=j_{max}$.
It suffices to prove that if $g_i$ is $L^2$ nonnegative functions
supported in $I_{k_i}$, $i=1,3,4$, and $g$ is a $L^2$ nonnegative
function supported in $I_{k_2}\times {I}_{j_2}$, then
\begin{eqnarray}\label{eq:l41dc21}
&&\int_{\R^3\cap \{\xi_1\cdot
\xi_2<0\}}g_1(\xi_1)g_3(\xi_3)g_4(\xi_4)g(\xi_1+\xi_3+\xi_4,
\Omega(\xi_1,\xi_3,\xi_4))d\xi_1d\xi_3d\xi_4\nonumber\\
&\les&
2^{j_2/2}2^{-k_3}\norm{g_1}_{L^2}\norm{g_4}_{L^2}\norm{g_3}_{L^2}\norm{g}_{L^2}.
\end{eqnarray}
As the case $j_4=j_{max}$, we get that the right-hand side of
\eqref{eq:l41dc21} is bounded by
\begin{eqnarray}\label{eq:l41dc22}
\sum_{l}\int_{\R^3}\chi_{l}(\xi_3+\xi_4)g_1(\xi_1)g_4(\xi_4)g_3(\xi_3)g(\xi_1+\xi_4+\xi_3,
\Omega(\xi_1,\xi_3,\xi_4))d\xi_1d\xi_4d\xi_3.
\end{eqnarray}
From the support properties of the functions $g_i,\ g$ and Lemma
\ref{esresonance} that in the integration area
\[|\Omega(\xi_1,\xi_3,\xi_4)|\sim|(\xi_1+\xi_4)(\xi_4+\xi_3)|\sim 2^{l+k_3},\]
We get that
\begin{equation}\label{eq:l41dc2si}
L_{max}\ges \ 2^{l+k_3}.
\end{equation}
By changing variable of integration $\xi_1'=\xi_1+\xi_3$,
$\xi_3'=\xi_3+\xi_4$, $\xi_4'=\xi_1+\xi_3+\xi_4$, we obtain that
\eqref{eq:l41dc22} is bounded by
\begin{eqnarray}\label{eq:l41dc23}
&&\sum_{l}\int_{|\xi_3'|\sim 2^l,|\xi_4'|\sim 2^{k_2},|\xi_1'|\sim 2^{k_3}}\chi_{l}(\xi_3')g_1(\xi_4'-\xi_3')g_3(\xi_1'+\xi_3'-\xi_4')g_4(\xi_4'-\xi_1')\nonumber\\
&&\quad g(\xi_4',
\Omega(\xi_4'-\xi_3',\xi_1'+\xi_3'-\xi_4',\xi_4'-\xi_1'))d\xi_1'd\xi_3'd\xi_4'.
\end{eqnarray}
Since in the integration area,
\begin{eqnarray}\label{eq:l41dc2jo}
&&\big|\frac{\partial}{\partial_{\xi_3'}}[\Omega(\xi_4'-\xi_3',\xi_1'+\xi_3'-\xi_4',\xi_4'-\xi_1')]\big|\nonumber\\
&=&|-\dr'(\xi_4'-\xi_3')+\dr'(\xi_1'+\xi_3'-\xi_4')|\sim 2^{k_3},
\end{eqnarray}
then we get from \eqref{eq:l41dc2jo} that \eqref{eq:l41dc23} is
bounded by
\begin{eqnarray}
&&\sum_{l}\int_{|\xi_3'|\sim 2^l}\chi_{l}(\xi_3')\norm{g_1}_{L^2}\norm{g_3}_{L^2}\nonumber\\
&&\quad \norm{g_4(\xi_4'-\xi_1')g(\xi_4',
\Omega(\xi_4'-\xi_3',\xi_1'+\xi_3'-\xi_4',\xi_4'-\xi_1'))}_{L^2_{\xi_1',\xi_4'}}d\xi_3'\nonumber\\
&\les&\sum_{l}2^{l/2}2^{-k_3/2}\norm{g_1}_{L^2}\norm{g_3}_{L^2}\norm{g_4}_{L^2}\norm{g}_{L^2}\nonumber\\
&\les&
2^{j_{max}/2}2^{-k_3}\norm{g_1}_{L^2}\norm{g_2}_{L^2}\norm{g_3}_{L^2}\norm{g}_{L^2},
\end{eqnarray}
where we used \eqref{eq:l41dc2si} in the last inequality. Therefore,
we complete the proof for part (d).
\end{proof}

We restate now Lemma \ref{l41} in a form that is suitable for the
trilinear estimates in the next sections.
\begin{corollary}\label{cor42}
Assume $\delta\geq c_0$. Let $k_i, j_i \in \Z_+$ and $N_i=2^{k_i},
L_i=2^{j_i}$ for $ i=1,2,3,4$. Assume $f_{k_i,j_i}\in L^2(\R\times
\R)$ are functions supported in $D_{k_i,j_i}$, $i=1,2,3.$

(a) For any $k_i, j_i\in \Z_+, i=1,2,3,4$,
\begin{eqnarray}
&&\norm{1_{D_{k_4,j_4}}(\xi,\tau)(f_{k_1,j_1}*f_{k_2,j_2}*f_{k_3,j_3})}_{L^2}\nonumber\\
&\leq&
C2^{(k_{min}+k_{thd})/2}2^{(j_{min}+j_{thd})/2}\prod_{i=1}^3\norm{f_{k_i,j_i}}_{L^2}.
\end{eqnarray}

(b) For any $k_i, j_i\in \Z_+, i=1,2,3,4$ with $N_{thd}\ll N_{sub}$.
If for some $i\in \{1,2,3,4\}$ such that
$(k_i,j_i)=(k_{thd},j_{max})$, then
\begin{eqnarray}
&&\norm{1_{D_{k_4,j_4}}(\xi,\tau)(f_{k_1,j_1}*f_{k_2,j_2}*f_{k_3,j_3})}_{L^2}\nonumber\\
&\leq&
C2^{(-k_{max}+k_{thd})/2}2^{(j_{1}+j_{2}+j_3+j_4)/2}2^{-j_{max}/2}\prod_{i=1}^3\norm{f_{k_i,j_i}}_{L^2},
\end{eqnarray}
else we have
\begin{eqnarray}
&&\norm{1_{D_{k_4,j_4}}(\xi,\tau)(f_{k_1,j_1}*f_{k_2,j_2}*f_{k_3,j_3})}_{L^2}\nonumber\\
&\leq&
C2^{(-k_{max}+k_{min})/2}2^{(j_{1}+j_{2}+j_3+j_4)/2}2^{-j_{max}/2}\prod_{i=1}^3\norm{f_{k_i,j_i}}_{L^2}.
\end{eqnarray}

(c) For any $k_i, j_i\in \Z_+, i=1,2,3,4$, with $N_{max}\gg 1$,
\begin{eqnarray}
&&\norm{1_{D_{k_4,j_4}}(\xi,\tau)(f_{k_1,j_1}*f_{k_2,j_2}*f_{k_3,j_3})}_{L^2}\nonumber\\
&\leq&
C2^{(j_{1}+j_{2}+j_3+j_4)/2}2^{-j_{max}/2}\prod_{i=1}^3\norm{f_{k_i,j_i}}_{L^2}.
\end{eqnarray}

(d) If $N_{min}\ll N_{max}$ and $N_{max}\gg 1$, then
\begin{eqnarray}
&&\norm{1_{D_{k_4,j_4}}(\xi,\tau)(f_{k_1,j_1}*f_{k_2,j_2}*f_{k_3,j_3})}_{L^2}\nonumber\\
&\leq&
C2^{(j_{1}+j_{2}+j_3+j_4)/2}2^{-k_{max}/2}\prod_{i=1}^3\norm{f_{k_i,j_i}}_{L^2}.
\end{eqnarray}
\end{corollary}
\begin{proof}
Clearly, we have
\begin{eqnarray}
&&\norm{1_{D_{k_4,j_4}}(\xi,\tau)(f_{k_1,j_1}*f_{k_2,j_2}*f_{k_3,j_3})(\xi,\tau)}_{L^2}\nonumber\\
&=&\sup_{\norm{f}_{L^2}=1}\aabs{\int_{D_{k_4,j_4}} f\cdot
f_{k_1,j_1}*f_{k_2,j_2}*f_{k_3,j_3} d\xi d\tau}.
\end{eqnarray}
Let $f_{k_4,j_4}=1_{D_{k_4,j_4}}\cdot f$, and then
$f_{k_i,j_i}^\sharp(\xi,\mu)=f_{k_i,j_i}(\xi,\mu+\dr(\xi))$,
$i=1,2,3,4$. The functions $f_{k_i,j_i}^\sharp$ are supported in
$I_{k_i}\times \cup_{|m|\leq 3} {I}_{j_i+m}$,
$\norm{f_{k_i,j_i}^\sharp}_{L^2}=\norm{f_{k_i,j_i}}_{L^2}$. Using
simple changes of variables, we get
\[\int_{D_{k_4,j_4}} f\cdot
f_{k_1,j_1}*f_{k_2,j_2}*f_{k_3,j_3} d\xi d\tau =
J(f_{k_1,j_1}^\sharp,f_{k_2,j_2}^\sharp,f_{k_3,j_3}^\sharp,f_{k_4,j_4}^\sharp).\]
Then Corollary \ref{cor42} follows from Lemma \ref{l41}.
\end{proof}

\section{Trilinear Estimate}

In this section we devote to prove the trilinear estimates. We
divide it into several cases. The first case is $low\times high
\rightarrow high$ interactions.

\begin{proposition}\label{p51} Assume $\delta\geq c_0$. Let $k_i\in \Z_+, N_i=2^{k_i}, i=1,2,3,4.$ Assume $N_3\gg 1$, $N_4\sim N_3$, $
N_1 \sim N_2\ll N_3$, and $f_{k_i}\in Z_{k_i}$ with
$\ft^{-1}(f_{k_i})$ compactly supported (in time) in $I$  with
$|I|\les 1$, $i=1,2,3$. Then
\begin{eqnarray}\label{eq:p51}
2^{k_4}\norm{\chi_{k_4}(\xi)(\tau-\dr(\xi)+i)^{-1}f_{k_1}*f_{k_2}*f_{k_3}}_{Z_{k_4}}\les
\  2^{(k_1+k_2)/2}\prod_{i=1}^3\norm{f_{k_i}}_{Z_{k_i}}.
\end{eqnarray}
\end{proposition}
\begin{proof}
We first divide it into three parts. Fix $M\gg 1$, then
\begin{eqnarray*}
&&2^{k_4}\norm{\chi_{k_4}(\xi)(\tau-\dr(\xi)+i)^{-1}f_{k_1}*f_{k_2}*f_{k_3}}_{Z_{k_4}}\\
&\leq&2^{k_4}\norm{\chi_{k_4}(\xi)\eta_{\leq k_4-1}(\tau-\dr(\xi))(\tau-\dr(\xi)+i)^{-1}f_{k_1}*f_{k_2}*f_{k_3}}_{Z_{k_4}}\\
&&+2^{k_4}\norm{\chi_{k_4}(\xi)\eta_{[k_4,2k_4+M]}(\tau-\dr(\xi))(\tau-\dr(\xi)+i)^{-1}f_{k_1}*f_{k_2}*f_{k_3}}_{Z_{k_4}}\\
&&+2^{k_4}\norm{\chi_{k_4}(\xi)\eta_{\geq
2k_4+M+1}(\tau-\dr(\xi))(\tau-\dr(\xi)+i)^{-1}f_{k_1}*f_{k_2}*f_{k_3}}_{Z_{k_4}}\\
&=&I+II+III.
\end{eqnarray*}

We consider first the contribution of I. Using $Y_k$ norm, then we
get from Lemma \ref{basicproperties} (c) and Lemma \ref{embedding}
that
\begin{eqnarray*}
I&\leq& 2^{k_4}\norm{\chi_{k_4}(\xi)\eta_{\leq
k_4-1}(\tau-\dr(\xi))(\tau-\dr(\xi)+i)^{-1}f_{k_1}*f_{k_2}*f_{k_3}}_{Y_{k_4}}\\
&\les& 2^{k_4/2}\norm{\ft^{-1}[f_{k_1}*f_{k_2}*f_{k_3}]}_{L_x^1L_t^2}\\
&\les& 2^{k_4/2}\norm{\ft^{-1}(f_{k_3})}_{L_x^\infty L_t^2}\norm{\ft^{-1}(f_{k_2})}_{L_x^2L_t^\infty}\norm{\ft^{-1}(f_{k_1})}_{L_x^2L_t^\infty}\\
&\les& 2^{(k_1+k_2)/2}\prod_{i=1}^3\norm{f_{k_i}}_{Z_{k_i}},
\end{eqnarray*}
which is \eqref{eq:p51} as desired.

For the contribution of II, we use $X_k$ norm. Then we get from
Lemma \ref{embedding} that
\begin{eqnarray*}
II&\leq&
2^{k_4}\norm{\chi_{k_4}(\xi)\eta_{[k_4,2k_4+M]}(\tau-\dr(\xi))(\tau-\dr(\xi)+i)^{-1}f_{k_1}*f_{k_2}*f_{k_3}}_{Z_{k_4}}\\
&\leq&\sum_{k_4\leq j \leq
2k_4+M}2^{k_4}2^{-j/2}\norm{1_{D_{k_4,j}}(\xi,\tau)f_{k_1}*f_{k_2}*f_{k_3}}_{L^2}\\
&\leq&\sum_{k_4\leq j \leq
2k_4+M}2^{k_4}2^{-j/2}\norm{\ft^{-1}(f_{k_3})}_{L_x^\infty L_t^2}\norm{\ft^{-1}(f_{k_1})}_{L_x^4 L_t^\infty}\norm{\ft^{-1}(f_{k_2})}_{L_x^4 L_t^\infty}\\
&\les& 2^{(k_1+k_2)/4}\prod_{i=1}^3\norm{f_{k_i}}_{Z_{k_i}}.
\end{eqnarray*}

Finally we consider the contribution of III. Let
$f_{k_i,j_i}(\xi,\tau)=f_{k_i}(\xi,\tau)\eta_{j_i}(\tau-\dr(\xi))$,
$j_i\geq 0$, $i=1,2,3$. Using $X_k$ norm, we get
\begin{eqnarray*}
III\leq \sum_{j_4\geq 2k_4+M+1}\sum_{j_1,j_2,j_3\geq
0}\norm{1_{D_{k_4,j_4}}(\xi,\tau)f_{k_1,j_1}*f_{k_2,j_2}*f_{k_3,j_3}}_{L^2}.
\end{eqnarray*}
Since in the area $\{|\xi_i|\in I_{k_i}, i=1,2,3\}$, we have
$|\Omega(\xi_1,\xi_2,\xi_3)|\ll 2^{2k_4+M}$. By checking the support
properties of $f_{k_i,j_i}$, we get $L_{max}\sim L_{sub}$. We
consider only  the worst case $L_4\sim L_3\ges L_1,L_2$, since the
other cases are better. It follows from Corollary \ref{cor42} and
Lemma \ref{elemes} (b) that
\begin{eqnarray*}
III&\les& \sum_{j_3\geq 2k_4+M+1}\sum_{j_1,j_2\geq
0}2^{(j_1+j_2)/2}2^{(k_1+k_2)/2}\norm{f_{k_1,j_1}}_{L^2}\norm{f_{k_2,j_2}}_{L^2}\norm{f_{k_3,j_3}}_{L^2}\\
&\les&\sum_{j_3\geq 2k_4+M+1}2^{j_3/4}2^{k_3-j_3}2^{(k_1+k_2)/2}2^{j_3-k_3}\norm{f_{k_1}}_{Z_{k_1}}\norm{f_{k_2}}_{Z_{k_2}}\norm{f_{k_3,j_3}}_{L^2}\\
&\les& \sum_{j_3\geq
2k_4+M+1}2^{k_3-\frac{3}{4}j_3}2^{(k_1+k_2)/2}\norm{f_{k_1}}_{Z_{k_1}}\norm{f_{k_2}}_{Z_{k_2}}\norm{f_{k_3}}_{Z_{k_3}}\\
&\les&
2^{(k_1+k_2)/4}\norm{f_{k_1}}_{Z_{k_1}}\norm{f_{k_2}}_{Z_{k_2}}\norm{f_{k_3}}_{Z_{k_3}}.
\end{eqnarray*}
Therefore, we complete the proof of the proposition.
\end{proof}

This proposition suffices to control $high\times low$ interaction in
the case that the two low frequences is comparable. However, for the
case that the two low frequences is not comparable, we will need an
improvement.

\begin{proposition} Assume $\delta\geq c_0$. Let $k_i\in \Z_+, N_i=2^{k_i}, i=1,2,3,4.$ Assume $N_3\gg 1$, $N_4\sim N_3$, $
N_1 \ll N_2\ll N_3$, and $f_{k_i}\in Z_{k_i}$ with
$\ft^{-1}(f_{k_i})$ compactly supported (in time) in $I$ with
$|I|\les 1$, $i=1,2,3$. Then
\begin{eqnarray*}
2^{k_4}\norm{\chi_{k_4}(\xi)(\tau-\dr(\xi)+i)^{-1}f_{k_1}*f_{k_2}*f_{k_3}}_{Z_{k_4}}\les
2^{(k_1+k_2)/4}\prod_{i=1}^3\norm{f_{k_i}}_{Z_{k_i}}.
\end{eqnarray*}
\end{proposition}
\begin{proof}
We first observe that in this case we get from Lemma
\ref{esresonance} that
\begin{equation}
|\Omega(\xi_1,\xi_2,\xi_3)|\sim 2^{k_3+k_2}.
\end{equation}
Dividing it into three parts and fixing an integer $M$ such that
$N_2\geq M\gg 1$, we obtain
\begin{eqnarray*}
&&2^{k_4}\norm{\chi_{k_4}(\xi)(\tau-\dr(\xi)+i)^{-1}f_{k_1}*f_{k_2}*f_{k_3}}_{Z_{k_4}}\\
&\leq&2^{k_4}\norm{\chi_{k_4}(\xi)\eta_{\leq k_4-1}(\tau-\dr(\xi))(\tau-\dr(\xi)+i)^{-1}f_{k_1}*f_{k_2}*f_{k_3}}_{Z_{k_4}}\\
&&+2^{k_4}\norm{\chi_{k_4}(\xi)\eta_{[k_4,2k_4+M]}(\tau-\dr(\xi))(\tau-\dr(\xi)+i)^{-1}f_{k_1}*f_{k_2}*f_{k_3}}_{Z_{k_4}}\\
&&+2^{k_4}\norm{\chi_{k_4}(\xi)\eta_{\geq
2k_4+M+1}(\tau-\dr(\xi))(\tau-\dr(\xi)+i)^{-1}f_{k_1}*f_{k_2}*f_{k_3}}_{Z_{k_4}}\\
&=&I+II+III.
\end{eqnarray*}
For the last two terms II, III,  we can use the same argument as for
II, III in the proof of Proposition \ref{p51}. We consider now the
first term I.
\begin{eqnarray*}
I&\leq&2^{k_4}\norm{\chi_{k_4}(\xi)\eta_{\leq
k_4-1}(\tau-\dr(\xi))(\tau-\dr(\xi)+i)^{-1}f_{k_1}*f_{k_2}*f_{k_3}^h}_{Z_{k_4}}\\
&&+2^{k_4}\norm{\chi_{k_4}(\xi)\eta_{\leq
k_4-1}(\tau-\dr(\xi))(\tau-\dr(\xi)+i)^{-1}f_{k_1}*f_{k_2}*f_{k_3}^l}_{Z_{k_4}}\\
&=&I_1+I_2,
\end{eqnarray*}
where
\[f_{k_3}^h=f_{k_3}(\xi,\tau)\eta_{\geq k_3+k_2-M}(\tau-\dr(\xi)),\ f_{k_3}^l=f_{k_3}(\xi,\tau)\eta_{\leq k_3+k_2-M+1}(\tau-\dr(\xi)).\]

For the contribution of $I_1$, we observe first that from the
support of $f_{k_3}^h$ and the definition of $Y_{k}$, one easily get
that
\begin{eqnarray*}
\norm{f_{k_3}^h}_{X_{k_3}}\les \norm{f_{k_3}}_{Z_{k_3}}.
\end{eqnarray*}
Thus from the definition of $Y_k$, and from H\"older's inequality,
Lemma \ref{basicproperties}, Lemma \ref{embedding} (b),  we get
\begin{eqnarray*}
I_1&\les&2^{k_4}\norm{\chi_{k_4}(\xi)\eta_{\leq
k_4-1}(\tau-\dr(\xi))(\tau-\dr(\xi)+i)^{-1}f_{k_1}*f_{k_2}*f_{k_3}^h}_{Y_{k_4}}\\
&\les&2^{k_4/2}\norm{\ft^{-1}[f_{k_1}*f_{k_2}*f_{k_3}^h]}_{L_x^1L_t^2}\\
&\les&2^{k_4/2}\norm{\ft^{-1}(f_{k_3}^h)}_{L_x^2L_t^2}\norm{\ft^{-1}(f_{k_1})}_{L_x^4L_t^\infty}\norm{\ft^{-1}(f_{k_2})}_{L_x^4L_t^\infty}\\
&\les&2^{k_4/2}2^{(k_1+k_2)/4}\norm{f_{k_3}^h}_{L^2}\norm{f_{k_1}}_{Z_{k_1}}\norm{f_{k_2}}_{Z_{k_2}}.
\end{eqnarray*}
Then from the fact that
\begin{eqnarray*}
2^{k_4/2}\norm{f_{k_3}^h}_{L^2}&\les& \sum_{j\geq
k_4+k_2-10}2^{k_4/2}\norm{f_{k_3}^h\eta_{j}(\tau-\dr(\xi))}_{L^2}\\
&\les&\norm{f_{k_3}^h}_{X_{k_3}}\les \norm{f_{k_3}}_{Z_{k_3}}
\end{eqnarray*}
we conclude the proof for $I_1$.

We consider now the contribution of $I_2$ where $\beta_{k,j}$ plays
crucial roles. Let
$f_{k_i,j_i}(\xi,\tau)=f_{k_i}(\xi,\tau)\eta_{j_i}(\tau-\dr(\xi))$,
$j_i\geq 0$, $i=1,2,3$. Using $X_k$ norm, we get
\begin{eqnarray*}
I_2\leq \sum_{j_4\leq k_4-1}\sum_{j_3\leq k_4+k_2-M,j_1,j_2\geq
0}2^{k_4-j_4/2}\norm{1_{D_{k_4,j_4}}f_{k_1,j_1}*f_{k_2,j_2}*f_{k_3,j_3}}_{L^2}.
\end{eqnarray*}
From the support properties, we get that
$1_{D_{k_4,j_4}}(\xi,\tau)f_{k_1,j_1}*f_{k_2,j_2}*f_{k_3,j_3}\equiv
0$ unless
\begin{eqnarray*}
\left \{
\begin{array}{l}
L_1\sim L_2 \ges N_3N_2; \mbox{ or }\\
L_2\ll L_1 \sim N_3N_2; \mbox{ or } L_1\ll L_2 \sim N_3N_2.
\end{array}
\right.
\end{eqnarray*}
If $L_1\sim L_2 \ges N_3N_2$, it follows from Corollary \ref{cor42}
(b) and Lemma \ref{elemes} (b) that
\begin{eqnarray*}
I_2&\les& \sum_{j_3,j_4\leq 2k_4}\sum_{j_1,j_2\geq
0}2^{k_4}2^{-j_4/2}2^{(j_2+j_3+j_4)/2}2^{-k_3/2}2^{k_1/2}\prod_{i=1}^3\norm{f_{k_i,j_i}}_2\\
&\les& \sum_{j_1,j_2\geq 0}k_4^2
2^{j_2/2}2^{k_3/2}2^{k_1/2}\prod_{i=1}^2\norm{f_{k_i,j_i}}_2
\norm{f_{k_3}}_{Z_{k_3}}\\
&\les& \sum_{j_1,j_2\geq
0}k_4^2 2^{j_2/2}2^{k_3/2}2^{k_1/2}2^{\frac{4}{5}k_1+\frac{4}{5}k_2-\frac{9}{10}j_1-\frac{9}{10}j_2}\prod_{i=1}^2(2^{\frac{9}{10}j_i-\frac{4}{5}k_i}\norm{f_{k_i,j_i}}_2) \norm{f_{k_3}}_{Z_{k_3}}\\
&\les&2^{(k_1+k_2)/4}\prod_{i=1}^3\norm{f_{k_i}}_{Z_{k_i}},
\end{eqnarray*}
which is acceptable. If $L_1\ll L_2 \sim N_3N_2$, it follows from
Corollary \ref{cor42} (b) that
\begin{eqnarray*}
I_2&\les& \sum_{j_1,j_3,j_4\leq 2k_4}\sum_{j_2\geq
0}2^{k_4}2^{(j_1+j_3)/2}2^{-k_3/2}2^{k_2/2}\prod_{i=1}^3\norm{f_{k_i,j_i}}_2\\
&\les&
k_4^32^{k_4}2^{\frac{4}{5}k_2-\frac{9}{10}j_2}2^{-k_3/2}2^{k_2/2}\prod_{i=1}^3\norm{f_{k_i}}_{Z_{k_i}}\les
2^{(k_1+k_2)/4}\prod_{i=1}^3\norm{f_{k_i}}_{Z_{k_i}}.
\end{eqnarray*}
The other case can be handled in the same way. Therefore, we
complete the proof of the proposition.
\end{proof}

\begin{proposition} Assume $\delta\geq c_0$.  Let $k_i\in \Z_+, N_i=2^{k_i}, i=1,2,3,4.$ Assume $N_3\gg 1$, $N_1\ll N_2\sim N_3\sim N_4$, and $f_{k_i}\in Z_{k_i}$ with
$\ft^{-1}(f_{k_i})$ compactly supported (in time) in $I$ with
$|I|\les 1$, $i=1,2,3$. Then
\begin{eqnarray*}
2^{k_4}\norm{\chi_{k_4}(\xi)(\tau-\dr(\xi)+i)^{-1}f_{k_1}*f_{k_2}*f_{k_3}}_{Z_{k_4}}\les
2^{(k_1+k_2)/4}\prod_{i=1}^3\norm{f_{k_i}}_{Z_{k_i}}.
\end{eqnarray*}
\end{proposition}
\begin{proof}
We first observe that this case corresponds to an integration in the
area $\{|\xi_i|\in I_{k_i},\ i=1,2,3\}\cap \{|\xi_1+\xi_2+\xi_3|\in
I_{k_4}\}$, where we have from Lemma \ref{esresonance} that
\begin{equation}
|\Omega(\xi_1,\xi_2,\xi_3)|\sim 2^{2k_3}.
\end{equation}
Let
$f_{k_i,j_i}(\xi,\tau)=f_{k_i}(\xi,\tau)\eta_{j_i}(\tau-\dr(\xi))$,
$j_i\geq 0$, $i=1,2,3$. Using $X_k$ norm, we get
\begin{eqnarray}\label{eq:p531}
&&2^{k_4}\norm{\chi_{k_4}(\xi)(\tau-\dr(\xi)+i)^{-1}f_{k_1}*f_{k_2}*f_{k_3}}_{Z_{k_4}}\nonumber
\\ &\les& \sum_{j_i\geq
0}2^{k_4}2^{-j_4/2}(1+2^{(j_4-2k_4)/2})\norm{1_{D_{k_4,j_4}}(\xi,\tau)f_{k_1,j_1}*f_{k_2,j_2}*f_{k_3,j_3}}_{L^2}.
\end{eqnarray}
From the support properties of the functions $f_{k_i,j_i}$,
$i=1,2,3$, it is easy to see that
$1_{D_{k_4,j_4}}(\xi,\tau)f_{k_1,j_1}*f_{k_2,j_2}*f_{k_3,j_3}\equiv
0$ unless
\begin{eqnarray*}
\left \{
\begin{array}{l}
L_{max}\sim L_{sub} \ges N_3^2; \mbox{ or }\\
L_{sub}\ll L_{max}\sim N_3^2.
\end{array}
\right.
\end{eqnarray*}

If $L_{max}\sim j_{sub} \ges N_3^2$, it follows from Corollary
\ref{cor42} (a) that the right-hand side of \eqref{eq:p531} is
bounded by
\begin{eqnarray}\label{eq:p53c1}
\sum_{j_i\geq
0}2^{k_4}2^{(j_1+j_2+j_3)/2}(1+2^{(j_4-2k_4)/2})2^{-(j_{sub}+j_{max})/2}2^{(k_1+k_2)/2}\prod_{i=1}^3\norm{f_{k_i,j_i}}_2.
\end{eqnarray}
It suffices to consider the worst case $j_3,j_4=j_{max}, j_{sub}$.
We get from Lemma \ref{basicproperties} (b) that  \eqref{eq:p53c1}
is bounded by
\begin{eqnarray}
\sum_{j_3\geq
2k_3-10}2^{k_4}2^{-\frac{3}{4}j_3}2^{(k_1+k_2)/2}\prod_{i=1}^3\norm{f_{k_i}}_{Z_{k_i}}\les2^{(k_1+k_2)/4}\prod_{i=1}^3\norm{f_{k_i}}_{Z_{k_i}}.
\end{eqnarray}

If $L_{sub}\ll L_{max}\sim N_3^2$, then from Corollary \ref{cor42}
(c) we get that the right side of \eqref{eq:p531} is bounded by
\begin{eqnarray*}
\sum_{j_i\geq
0}2^{k_4}2^{(j_1+j_2+j_3)/2}2^{-j_{max}/2}\prod_{i=1}^3\norm{f_{k_i,j_i}}_2\les
2^{(k_1+k_2)/4}\prod_{i=1}^3\norm{f_{k_i}}_{Z_{k_i}},
\end{eqnarray*}
where we used Lemma \ref{basicproperties} (b). Thus, we complete the
proof of the proposition.
\end{proof}

\begin{proposition}  Assume $\delta\geq c_0$.  Let $k_i\in \Z_+, N_i=2^{k_i}, i=1,2,3,4.$ Assume $N_3\gg 1$, $N_1\sim N_2\sim N_3\sim N_4$, and $f_{k_i}\in Z_{k_i}$ with
$\ft^{-1}(f_{k_i})$ compactly supported (in time) in $I$ with
$|I|\les 1$, $i=1,2,3$. Then
\begin{eqnarray*}
2^{k_4}\norm{\chi_{k_4}(\xi)(\tau-\dr(\xi)+i)^{-1}f_{k_1}*f_{k_2}*f_{k_3}}_{Z_{k_4}}\les
2^{k_4}\prod_{i=1}^3\norm{f_{k_i}}_{Z_{k_i}}.
\end{eqnarray*}
\end{proposition}
\begin{proof}
First we divide it into two parts. Fixing $M\gg 1$, then we have
\begin{eqnarray*}
&&2^{k_4}\norm{\chi_{k_4}(\xi)(\tau-\dr(\xi)+i)^{-1}f_{k_1}*f_{k_2}*f_{k_3}}_{Z_{k_4}}\\
&\les& 2^{k_4}\norm{\chi_{k_4}(\xi)\eta_{\leq 2k_4+M} (\tau-\dr(\xi))(\tau-\dr(\xi)+i)^{-1}f_{k_1}*f_{k_2}*f_{k_3}}_{Z_{k_4}}\\
&&+2^{k_4}\norm{\chi_{k_4}(\xi)\eta_{\geq 2k_4+M+1}
(\tau-\dr(\xi))(\tau-\dr(\xi)+i)^{-1}f_{k_1}*f_{k_2}*f_{k_3}}_{Z_{k_4}}\\
&=&I+II.
\end{eqnarray*}

We consider first the contribution of the first term $I$. Using the
$X_k$ norm and Lemma \ref{embedding} (a), then we get
\begin{eqnarray*}
I&\les& 2^{k_4}\sum_{j_4\geq
0}^{2k_4+20}2^{-j_4/2}\norm{1_{D_{k_4,j_4}}(\xi,\tau)f_{k_1}*f_{k_2}*f_{k_3}}_{L^2}\\
&\les&2^{k_4}\prod_{i=1}^3\norm{\ft^{-1}(f_{k_i})}_{L^6}\les
2^{k_4}\prod_{i=1}^3\norm{f_{k_i}}_{Z_{k_i}}.
\end{eqnarray*}

We consider now the contribution of the second term $II$. Let
$f_{k_i,j_i}(\xi,\tau)=f_{k_i}(\xi,\tau)\eta_{j_i}(\tau-\dr(\xi))$,
$j_i\geq 0$, $i=1,2,3$. Using the $X_k$ norm, we get
\begin{eqnarray}
II&\les&\sum_{j_4\geq 2k_4+20}\sum_{j_1,j_2,j_3\geq
0}\norm{1_{D_{k_4,j_4}}(\xi,\tau)f_{k_1,j_1}*f_{k_2,j_2}*f_{k_3,j_3}}_{L^2}.
\end{eqnarray}
Since in the area $\{|\xi_i|\in I_{k_i},\ i=1,2,3\}$ we have
$|\Omega(\xi_1,\xi_2,\xi_3)|\les 2^{2k_3}$, by checking the support
properties of the functions $f_{k_i,j_i}$, $i=1,2,3$, we get
$L_{max}\sim L_{sub}\gg N_3^2$. From symmetry, we assume
$j_3,j_4=j_{max}, j_{sub}$, then we get
\begin{eqnarray*}
II&\les&\sum_{j_4\geq 2k_4}\sum_{j_1,j_2,j_3\geq
0}2^{(j_1+j_2)/2}2^{k_3}2^{\frac{4}{5}k_3-\frac{9}{10}j_3}2^{\frac{9}{10}j_3-\frac{4}{5}k_3}\prod_{i=1}^3\norm{f_{k_i,j_i}}_2\\
&\les&\sum_{j_3\geq
2k_4}2^{2k_3}2^{-j_3/2}\prod_{i=1}^3\norm{f_{k_i}}_{Z_{k_i}}\les2^{k_4}\prod_{i=1}^3\norm{f_{k_i}}_{Z_{k_i}}.
\end{eqnarray*}
Therefore we complete the proof of the proposition.
\end{proof}

We consider now the case which corresponds to $high\times high $
interactions. This case is better than $high\times low$ interaction
case.
\begin{proposition} Assume $\delta\geq c_0$.  Let $k_i\in \Z_+,\ N_i=2^{k_i}, i=1,2,3,4.$ Assume $N_1\gg 1$, $N_4\ll N_1$, $N_3\les N_1\sim N_2$, and $f_{k_i}\in Z_{k_i}$ with $\ft^{-1}(f_{k_i})$ compactly
supported (in time) in $I$ with $|I|\les 1$, $i=1,2,3$. Then
\begin{eqnarray*}
2^{k_4}\norm{\eta_{k_4}(\xi)(\tau-\dr(\xi)+i)^{-1}f_{k_1}*f_{k_2}*f_{k_3}}_{Z_{k_4}}
\les {k^4_1}\prod_{i=1}^3\norm{f_{k_i}}_{Z_{k_i}}.
\end{eqnarray*}
\end{proposition}

\begin{proof}
Let
$f_{k_i,j_i}(\xi,\tau)=f_{k_i}(\xi,\tau)\eta_{j_i}(\tau-\dr(\xi))$,
$j_i\geq 0$, $i=1,2,3$. Using $X_k$ norm, then we get
\begin{eqnarray}\label{eq:hh}
&&2^{k_4}\norm{\chi_{k_4}(\xi)(\tau-\dr(\xi)+i)^{-1}f_{k_1}*f_{k_2}*f_{k_3}}_{Z_{k_4}}\nonumber\\
&\les& \sum_{j_i\geq
0}2^{k_4}2^{-j_4/2}(1+2^{(j_4-2{k_4})/2})\norm{1_{D_{k_4,j_4}}(\xi,\tau)f_{k_1,j_1}*f_{k_2,j_2}*f_{k_3,j_3}}_{L^2}.
\end{eqnarray}
If $L_{max}\les N_1^2$, then it follows from Corollary \ref{cor42}
(d) that the right side of \eqref{eq:hh} is bounded by
\begin{eqnarray*}
\sum_{j_i\geq
0}2^{k_4}(1+2^{(j_4-2{k_4})/2})2^{(j_1+j_2+j_3)/2}2^{-k_1}\prod_{i=1}^3\norm{f_{k_i,j_i}}_{L^2}\les{k^4_1}
\prod_{i=1}^3\norm{f_{k_i}}_{Z_{k_i}},
\end{eqnarray*}
where we used Lemma \ref{basicproperties} (b).

If $L_{max}\gg N_1^2$, then by checking the support properties, we
get $L_{max}\sim L_{sub}$. We consider only the worst case
$j_1,j_4=j_{max},j_{sub}$. It follows from Corollary \ref{cor42} (a)
and Lemma \ref{basicproperties} (b) that the right side of
\eqref{eq:hh} is bounded by
\begin{eqnarray*}
\sum_{j_i\geq
0}2^{k_4}2^{-j_4/2}(1+2^{(j_4-2{k_4})/2})2^{(j_2+j_3)/2}2^{k_4}\prod_{i=1}^3\norm{f_{k_i,j_i}}_{L^2}\les{k_1}\prod_{i=1}^3\norm{f_{k_i}}_{Z_{k_i}}.
\end{eqnarray*}
Therefore, we complete the proof of the proposition.
\end{proof}

The next proposition is used to control $low \times low$
interactions. This interaction is easy to control.

\begin{proposition} Assume $\delta\geq c_0$.  Let $k_i\in \Z_+, N_i=2^{k_i}, i=1,2,3,4.$ Assume $N_{max}\les 1$, and $f_{k_i}\in Z_{k_i}$, $i=1,2,3$. Then
\begin{equation}
2^{k_4}\norm{\eta_{k_4}(\xi)(\tau-\dr(\xi)+i)^{-1}f_{k_1}*f_{k_2}*f_{k_3}}_{Z_{k_4}}\les
\prod_{i=1}^3\norm{f_{k_i}}_{Z_{k_i}}.
\end{equation}
\end{proposition}
\begin{proof}
Let
$f_{k_i,j_i}(\xi,\tau)=f_{k_i}(\xi,\tau)\eta_{j_i}(\tau-\dr(\xi))$,
$j_i\geq 0$, $i=1,2,3$. Using $X_k$ norm, Corollary \ref{cor42} (a)
and Lemma \ref{elemes} (b), then we get
\begin{eqnarray*}
&&2^{k_4}\norm{\eta_{k_4}(\xi)(\tau-\dr(\xi)+i)^{-1}f_{k_1}*f_{k_2}*f_{k_3}}_{Z_{k_4}}\\
&\les& \sum_{j_i\geq
0}L_{min}^{1/2}L_{thd}^{1/2}N_{min}^{1/2}N_{thd}^{1/2}\prod_{i=1}^3\norm{f_{k_i,j_i}}_{L^2}\les
2^{(k_{min}+k_{thd})/2}\prod_{i=1}^3\norm{f_{k_i}}_{Z_{k_i}},
\end{eqnarray*}
since for the case $j_{max}\gg 1$ we have $L_{max}\sim L_{sub}$ by
checking the support properties of the functions $f_{k_i,j_i}$,
$i=1,2,3$.
\end{proof}

Finally we present two counterexamples as in \cite{Guo}. The first
one shows why we use a $l^1$-type $X^{s,b}$ structure. The other one
shows a logarithmic divergence if we only use $X_k$ which is the
reason for us applying $Y_k$ structure.

\begin{proposition}\label{countertrilinear}
Let $\delta\geq c_0$. Assume $k\geq M$. Then there exist $f_{1}\in
X_1,\ f_k\in X_k$ such that
\begin{eqnarray}
2^{k}\norm{\eta_k(\xi)(\tau-\dr(\xi)+i)^{-1}f_1*f_1*f_k}_{X_k}\ges \
k \norm{f_1}_{X_1}\norm{f_1}_{X_1}\norm{f_k}_{X_k}.
\end{eqnarray}
\end{proposition}
\begin{proof}
From the proof of Proposition \ref{p51}, we easily see that the
worst interaction comes from the case that largest frequency
component has a largest modulation. So we construct this case
explicitly. Let $I=[1/2,1]$, and take
\[f_1(\xi,\tau)=\chi_{I}(\xi)\eta_1(\tau-\dr(\xi)),\ f_k(\xi,\tau)=\chi_{I_k}(\xi)\eta_k(\tau-\dr(\xi)).\]
From definition, we easily get $\norm{f_1}_{X_1}\sim 1$ and
$\norm{f_k}_{X_k}\sim 2^{3k/2}$ and
\[2^{k}\norm{\eta_k(\xi)(\tau-\dr(\xi)+i)^{-1}f_1*f_1*f_k}_{X_k}\ges 2^{k}\sum_{j=0}^{k/2}2^{-j/2}\norm{1_{D_{k,j}}\cdot f_1*f_1*f_k}_{L_{\xi,\tau}^2}.\]
On the other hand, we have for $j\leq k/2$
\begin{eqnarray*}
&&1_{D_{k,j}}(\xi,\tau)\cdot f_1*f_1*f_k\\
&=&\int
f_1(\xi_1,\tau_1)f_2(\xi_2,\tau_2)f_k(\xi-\xi_1-\xi_2,\tau-\tau_1-\tau_2)d\xi_1d\xi_2d\tau_1d\tau_2\\
&=&\int
\chi_I(\xi_1)\chi_I(\xi_2)\eta_1(\tau_1)\eta_1(\tau_2)\chi_{I_k}(\xi-\xi_1-\xi_2)\\
&&\cdot \eta_k(
\tau-\tau_1-\tau_2-\dr(\xi_1)-\dr(\xi_2)-\dr(\xi-\xi_1-\xi_2))d\xi_1d\xi_2d\tau_1d\tau_2\\
&\ges&
\chi_{[\frac{2^{10}-1}{2^{10}}2^k,\frac{2^{10}+1}{2^{10}}2^k]}(\xi)\eta_j(\tau-\dr(\xi)).
\end{eqnarray*}
Therefore, we get
\begin{eqnarray}\label{eq:logdiv}
2^{k}\sum_{j=0}^{k/2}2^{-j/2}\norm{1_{D_{k,j}}\cdot
f_1*f_1*f_k}_{L_{\xi,\tau}^2}\ges \ k2^{3k/2},
\end{eqnarray}
which completes the proof of the proposition.
\end{proof}

\begin{proposition}\label{counterXsb}
For any $s\in \R$, there doesn't exists $b\in \R$ such that
\begin{equation}\label{eq:trilinearXsb}
\norm{\partial_x(uvw)}_{X^{s,b-1}}\les\
\norm{u}_{X^{s,b}}\norm{v}_{X^{s,b}}\norm{w}_{X^{s,b}}.
\end{equation}
\end{proposition}
\begin{proof}
It is easy to see that the counterexample in the proof of
Proposition \ref{countertrilinear} shows that
\eqref{eq:trilinearXsb} doesn't hold for $b=1/2$ with a $k^{1/2}$
divergence in \eqref{eq:logdiv}. We assume now $b\neq 1/2$. Using
Plancherel's equality, we get that \eqref{eq:trilinearXsb} is
equivalent to
\begin{eqnarray}\label{eq:trilinearXsbequiv}
&&\norm{\frac{\jb{\xi}^s\xi}{\jb{\tau-\dr(\xi)}^{1-b}}\int
\frac{u(\xi_1,\tau_1)}{\jb{\xi_1}^s \jb{\tau_1-\dr(\xi_1)}^b}\frac{v(\xi_2,\tau_2)}{\jb{\xi_2}^s \jb{\tau_2-\dr(\xi_2)}^b}\nonumber\\
&&\quad \cdot \frac{w(\xi-\xi_1-\xi_2,\tau-\tau_1-\tau_2)}{\jb{\xi-\xi_1-\xi_2}^s \jb{\tau-\tau_1-\tau_2-\dr(\xi-\xi_1-\xi_2)}^b}d\tau_1 d\tau_2d\xi_1d\xi_2}_{L_{\xi,\tau}^2}\nonumber\\
&&\les \  \norm{u}_{L^2}\norm{v}_{L^2}\norm{w}_{L^2}.
\end{eqnarray}
Fix any dyadic number $N\gg 1$. Let
\[A=\{1/2\leq \xi\leq 10,\ |\tau|\leq
1\} \mbox{ and } B=\{N/2\leq \xi\leq 2N,\ |\tau|\leq 2^{10}\}.\]
Take
\[u(\xi,\tau)=v(\xi,\tau)=\chi_{A}(\xi,\tau-\dr(\xi)),\ w(\xi,\tau)=\chi_{B}(\xi,\tau-\dr(\xi)).\]
We easily see that $\norm{u}_{L_2}=\norm{v}_{L_2}\sim 1$ and
$\norm{w}_{L^2}\sim N^{1/2}$. Denote
$f(\xi,\tau)=u*v*w(\xi,\tau+\dr(\xi))$. Then we have
\begin{eqnarray*}
&&f(\xi,\tau)\\
&=&\int
u(\xi_1,\tau_1)v(\xi_2,\tau_2)w(\xi-\xi_1-\xi_2,\tau+\dr(\xi)-\tau_1-\tau_2)d\xi_1d\xi_2d\tau_1d\tau_2\\
&=&\int \chi_{\leq
2^{10}}(\tau-\tau_1-\tau_2+\dr(\xi)-\dr(\xi-\xi_1-\xi_2)-\dr(\xi_1)-\dr(\xi_2))\\
&&\chi_A(\xi_1,\tau_1)\chi_A(\xi_2,\tau_2)\chi_{[N/2,2N]}(\xi-\xi_1-\xi_2)d\xi_1d\xi_2d\tau_1d\tau_2\\
&=&\int \chi_{\leq 2^{10}}(\tau-\tau_1-\tau_2+2(\xi_1+\xi_2)\xi+(\xi_1-\xi_2)^2-\dr(\xi_1)-\dr(\xi_2)+o(1))\\
&&\chi_A(\xi_1,\tau_1)\chi_A(\xi_2,\tau_2)\chi_{[N/2,2N]}(\xi-\xi_1-\xi_2)d\xi_1d\xi_2d\tau_1d\tau_2.
\end{eqnarray*}
Therefore, fixing $M\gg 1$, we get for any $(\xi, \tau) \in
[(M-1)N/M, (M+1)N/M]\times [-8N,-4N]$, then $\tau=-C_0\xi$ for some
$2\leq C_0\leq 9$ and
\begin{eqnarray*}
f(\xi,\tau)\ges
\int\chi_A(\xi_1,\tau_1)\chi_A(\xi_2,\tau_2)\chi_{|\xi_1+\xi_2-C_0|\les
N^{-1}}d\xi_1d\xi_2d\tau_1d\tau_2\ges N^{-1}.
\end{eqnarray*}
Thus we see that the left-hand side of \eqref{eq:trilinearXsbequiv}
is larger than $N^b$, while the right-hand side is $N^{1/2}$, which
implies $b< 1/2$.

Similarly, by taking $B'=\{N/2\leq \xi\leq 2N,\ N\leq |\tau|\leq
N\}$ as before, we obtain that $b>1/2$. Therefore we complete the
proof of the proposition.
\end{proof}

\section{Proof of Theorem \ref{t11}}

In this section we devote to prove Theorem \ref{t11} by using the
standard fixed-point machinery. From Duhamel's principle, we get
that the equation \eqref{eq:mFDF} is equivalent to the following
integral equation:
\begin{eqnarray}\label{eq:inteq}
u=\U(t)\phi+\int_0^t \U(t-t')(\partial_x(u^3)(t'))dt'.
\end{eqnarray}
We will mainly work on the following truncated version
\begin{eqnarray}\label{eq:truninteq}
u=\psi(t)\U(t)\phi+\psi(t)\int_0^t
\U(t-t')\partial_x[(\psi(t')u)^3]dt',
\end{eqnarray}
where $\psi(t)=\eta_0(t)$ is a smooth cut-off function. Then we
easily see that if $u$ is a solution to \eqref{eq:truninteq} on
$\R$, then $u$ solves \eqref{eq:inteq} on $t\in [-1,1]$. Our first
lemma is on the estimate for the linear solution.

\begin{lemma}\label{l61}
If $\delta\geq c_0, s\geq 0$ and $\phi \in {H}^s$ then
\begin{equation}
\norm{\psi(t)\cdot (\U(t)\phi)}_{F^{s}}\les  \norm{\phi}_{{H}^s}.
\end{equation}
\end{lemma}
\begin{proof}
A direct computation shows that
\begin{eqnarray*}
\ft[\psi(t)\cdot
(\U(t)\phi)](\xi,\tau)=\widehat{\phi}(\xi)\widehat{\psi}(\tau-\dr(\xi)).
\end{eqnarray*}
In view of definition, it suffices to prove that if $k\in \Z_+$ then
\begin{equation}\label{eq:l61}
\norm{\eta_k(\xi)\widehat{\phi}(\xi)\widehat{\psi}(\tau-\dr(\xi))}_{Z_k}\leq
C \norm{\eta_k(\xi)\widehat{\phi}(\xi)}_{L^2}.
\end{equation}
Indeed, from definition we have
\begin{eqnarray*}
\norm{\eta_k(\xi)\widehat{\phi}(\xi)\widehat{\psi}(\tau-\dr(\xi))}_{Z_k}&\leq&
\norm{\eta_k(\xi)\widehat{\phi}(\xi)\widehat{\psi}(\tau-\dr(\xi))}_{X_k}\\
&\leq&C\sum_{j=0}^\infty 2^{j}
\norm{\eta_k(\xi)\widehat{\phi}(\xi)}_{L^2}
\norm{\eta_j(\tau)\widehat{\psi}(\tau)}_{L^2}\\
&\leq&C\norm{\eta_k(\xi)\widehat{\phi}(\xi)}_{L^2},
\end{eqnarray*}
which is \eqref{eq:l61} as desired.
\end{proof}

Next lemma is on the estimate for the retarded linear term. We will
follow the method in \cite{IK} to prove it.

\begin{lemma}\label{l62}
If $\delta\geq c_0,\ s\geq 0$ and $u \in \Sch(\R\times \R)$ then
\begin{equation}
\normo{\psi(t)\cdot \int_0^t\U(t-s)(u(s))ds}_{F^{s}}\leq C
\norm{u}_{N^{s}}.
\end{equation}
\end{lemma}
\begin{proof}
A straightforward computation shows that
\begin{eqnarray*}
&&\ft\left[\psi(t)\cdot
\int_0^t\U(t-s)(u(s))ds\right](\xi,\tau)\\
&=&c\int_\R
\ft(u)(\xi,\tau')\frac{\widehat{\psi}(\tau-\tau')-\widehat{\psi}(\tau-\dr(\xi))}{\tau'-\dr(\xi)}d\tau'.
\end{eqnarray*}
For $k\in \Z_+$ let
$f_k(\xi,\tau')=\ft(u)(\xi,\tau')\eta_k(\xi)(\tau'-\dr(\xi)+i)^{-1}$.
For $f_k\in Z_k$ let
\begin{eqnarray*}
T(f_k)(\xi,\tau)=\int_\R
f_k(\xi,\tau')\frac{\widehat{\psi}(\tau-\tau')-\widehat{\psi}(\tau-\dr(\xi))}{\tau'-\dr(\xi)}(\tau'-\dr(\xi)+i)d\tau'.
\end{eqnarray*}
In view of the definitions, it suffices to prove that
\begin{equation}
\norm{T}_{Z_k\rightarrow Z_k}\leq C \mbox{ uniformly in } k\in Z_+,
\end{equation}
which follows from the slightly modified proof of Lemma 5.2 in
\cite{IK}. We omit the details.
\end{proof}

We prove a trilinear estimate in the following proposition which is
an important component for using fixed-point argument.
\begin{proposition}\label{p63}
Assume $\delta\geq c_0$. Let $s\geq 1/2$. Then
\begin{eqnarray*}
\norm{\partial_x(\psi(t)^3uvw)}_{N^{s}}&\les&
\norm{u}_{F^{s}}\norm{v}_{F^{1/2}}\norm{w}_{F^{1/2}}\\
&&+\norm{u}_{F^{1/2}}\norm{v}_{F^{s}}\norm{w}_{F^{1/2}}+\norm{u}_{F^{1/2}}\norm{v}_{F^{1/2}}\norm{w}_{F^{s}}.
\end{eqnarray*}
\end{proposition}
\begin{proof}
In view of definition, we get
\begin{eqnarray*}
\norm{\partial_x(\psi(t)^3uvw)}_{N^{s}}^2=\sum_{k_4=0}^{\infty}2^{2sk_4}\norm{\eta_{k_4}(\xi)(\tau-\dr(\xi)+i)^{-1}\ft
(\partial_x(\psi(t)^3uvw))}_{Z_{k_4}}^2.
\end{eqnarray*}
For $k_1, k_2, k_3\in \Z_+$, setting
$f_{k_1}=\eta_{k_1}(\xi)\ft(\psi(t)u)(\xi,\tau)$,
$f_{k_2}=\eta_{k_2}(\xi)\ft(\psi(t)v)(\xi,\tau)$, and
$f_{k_3}=\eta_{k_3}(\xi)\ft(\psi(t)w)(\xi,\tau)$,  then we get
\begin{eqnarray*}
&&2^{k_4}\norm{\eta_{k_4}(\xi)(\tau-\dr(\xi)+i)^{-1}\ft
(\psi(t)^3uvw)}_{Z_{k_4}}\nonumber\\
&&\les \sum_{k_1,k_2,k_3\in
\Z_+}2^{k_4}\norm{\eta_{k_4}(\xi)(\tau-\dr(\xi)+i)^{-1}f_{k_1}*f_{k_2}*f_{k_3}}_{Z_{k_4}}.
\end{eqnarray*}
From symmetry it suffices to bound
\begin{eqnarray*}
\sum_{0\leq k_1\leq k_2\leq
k_3}2^{k_4}\norm{\eta_{k_4}(\xi)(\tau-\dr(\xi)+i)^{-1}f_{k_1}*f_{k_2}*f_{k_3}}_{Z_{k_4}}.
\end{eqnarray*}
Setting $N_i=2^{k_i}, i=1,2,3,4$, we get
\begin{eqnarray}\label{eq:trilinear}
&&\sum_{k_1\leq k_2\leq
k_3}2^{k_4}\norm{\eta_{k_4}(\xi)(\tau-\dr(\xi)+i)^{-1}f_{k_1}*f_{k_2}*f_{k_3}}_{Z_{k_4}}\nonumber\\
&\leq& \sum_{j=1}^6\sum_{(k_1,k_2,k_3,k_4)\in A_j}
2^{k_4}\norm{\eta_{k_4}(\xi)(\tau-\dr(\xi)+i)^{-1}f_{k_1}*f_{k_2}*f_{k_3}}_{Z_{k_4}},
\end{eqnarray}
where we denote
\begin{eqnarray*}
&&A_1=\{0\leq N_1\leq N_2\ll N_3, N_3\gg 1, N_4\sim N_3, N_1\sim N_2 \};\\
&&A_2=\{0\leq N_1\leq N_2\ll N_3, N_3\gg 1, N_4\sim N_3, N_1\ll N_2 \};\\
&&A_3=\{0\leq N_1\leq N_2\leq N_3, N_3\gg 1,
N_4\sim N_3\sim N_2, N_1\ll N_2 \};\\
&&A_4=\{N_1\sim N_2\sim N_3 \sim N_4, N_3\gg 1 \};\\
&&A_5=\{0\leq N_1\leq N_2\leq N_3, N_4\ll N_3, N_3\gg 1,
N_2\sim N_3\};\\
&&A_6=\{\max(N_3,N_4)\les 1\}.
\end{eqnarray*}
We will apply Proposition 5.1-5.6 obtained in the last section to
bound the six terms in \eqref{eq:trilinear}. For example, for the
first term, from Proposition 5.1, we have
\begin{eqnarray*}
&&\normb{2^{sk_4}\sum_{k_i\in
A_1}2^{k_4}\norm{\eta_{k_4}(\xi)(\tau-\dr(\xi)+i)^{-1}f_{k_1}*f_{k_2}*f_{k_3}}_{Z_{k_4}}}_{l_{k_4}^2}\\
&\leq& C\normb{2^{sk_4}\sum_{k_i\in
A_1}2^{(k_1+k_2)/2}\norm{f_{k_1}}_{Z_{k_1}}\norm{f_{k_2}}_{Z_{k_2}}\norm{f_{k_3}}_{Z_{k_3}}}_{l_{k_4}^2}\\
&\leq& \norm{u}_{F^{1/2}}\norm{v}_{F^{1/2}}\norm{w}_{F^s}.
\end{eqnarray*}
For the other terms we can handle them in the similar ways.
Therefore we complete the proof of the proposition.
\end{proof}

Now we prove Theorem \ref{t11}. To begin with, we renormalize the
data a bit via scaling. By the scaling \eqref{eq:scaling}, we see
that if $s\geq 1/2$
\begin{eqnarray*}
&&\norm{\phi_\lambda}_{L^2}=\norm{\phi}_{L^2},\\
&&\norm{\phi_\lambda}_{\dot{H}^s}=\lambda^{-s}\norm{\phi}_{\dot{H}^s}.
\end{eqnarray*}
From the assumption $\norm{\phi}_{L^2}\ll 1$, thus we can first
restrict ourselves to considering \eqref{eq:mBO} with data $\phi$
satisfying
\begin{equation}
\norm{\phi}_{H^s}=r\ll 1.
\end{equation}
This indicates the reason why we assume that $\norm{\phi}_{L^2}\ll
1$.

Define the operator
\begin{eqnarray*}
\Phi_{\phi}(u)=\psi(t)\U(t)\phi+\psi(t)\int_0^t
\U(t-t')(\partial_x((\psi(t')u)^3)(t'))dt',
\end{eqnarray*}
and we will prove that $\Phi_\phi (\cdot)$ is a contraction  mapping
from
\begin{equation}
{\mathcal{B}}=\{w\in F^s:\ \norm{w}_{F^s}\leq 2cr\}
\end{equation}
into itself. From Lemma \ref{l61}, \ref{l62} and Proposition
\ref{p63} we get if $w\in \mathcal{B}$, then
\begin{eqnarray}
\norm{\Phi_\phi(w)}_{F^s}&\leq&
c\norm{\phi}_{H^s}+\norm{\partial_x(\psi(t)^3w^3(\cdot, t))}_{N^s}\nonumber\\
&\leq& cr+c\norm{w}_{F^s}^3\leq cr+c(2cr)^3\leq 2cr,
\end{eqnarray}
provided that $r$ satisfies $8c^3r^2\leq 1/2$. Similarly, for $w,
h\in \mathcal{B}$
\begin{eqnarray}
\norm{\Phi_\phi(w)-\Phi_\phi(h)}_{F^s}
&\leq& c\normo{ \psi(t)\int_0^t \partial_x[\psi^3(\tau)(w^3(\tau)-h^3(\tau))]d\tau}_{F^s}\nonumber\\
&\leq&c(\norm{w}_{F^s}^2+\norm{h}_{F^s}^2)\norm{w-h}_{F^s}\nonumber\\
&\leq&8c^3r^2\norm{w-h}_{F^s}\leq \frac{1}{2}\norm{w-h}_{F^s}.
\end{eqnarray}
Thus $\Phi_\phi(\cdot)$ is a contraction. Therefore, there exists a
unique $u\in \mathcal{B}$ such that
\begin{eqnarray*}
u=\psi(t)W(t)\phi+\psi(t)\int_0^t
W(t-t')(\partial_x[(\psi(t')u)^3](t'))dt'.
\end{eqnarray*}
Hence $u$ solves the integral equation \eqref{eq:inteq} in the time
interval $[-1,1]$.

Part (c) of Theorem \ref{t11} follows from the scaling
\eqref{eq:scaling}, Lemma \ref{embedding} and Proposition \ref{p63}.
Pard (d) follows from the standard argument. We prove now part (b).
It is easy to see that the energy methods as in \cite{ABFS} show
local well-posedness for Eq. \eqref{eq:mFDF} in $H^s$ for $s>3/2$.
One may improve this to $H^1$, using the methods in \cite{KK}.
According to Theorem 1.2 in \cite{KK}, it suffices to prove that if
$s> 1$ then
\[\partial_x u \in L^4_{t\in [0,T]}L^\infty_x.\] Indeed, this follows
from the fact that $u \in F^s(T)$ and $(4,\infty)$ is an admissible
pair and Lemma \ref{embedding}. Therefore, we complete the proof of
Theorem \ref{t11}.

\section{Ill-posedness Result}

In this section we will prove that the solution map of Eq.
\eqref{eq:mFDF} is not $C^3$ differentiable at origin in $H^s$ if
$s<1/2$, closely following the method in \cite{MR5, MR4}. Thus we
see $H^{1/2}$ is the critical regularity for which one can get
wellposedness by fixed point argument. Following standard fixed
point argument, one need to find the Banach space $X^s\subset
C([0,T];H^s)$ such that it verifies
\begin{eqnarray}
\norm{\U(t)u_0}_{X^s}&\les& \norm{u_0}_{H^s},\label{eq:freesol}\\
\normo{\int_0^t \U(t-\tau)\partial_x
(u_1u_2u_3)(\tau)d\tau}_{X^s}&\les&
\norm{u_1}_{X^s}\norm{u_2}_{X^s}\norm{u_3}_{X^s}.\label{eq:retardes}
\end{eqnarray}
In particular, if we set $u_i=\U(t)\phi_i, \ i=1,2,3$, then we can
obtain from \eqref{eq:freesol} and \eqref{eq:retardes} that for
$0<t<T$,
\begin{eqnarray}\label{eq:trilinear}
\normo{\int_0^t \U(t-\tau)\partial_x
(\prod_{i=1}^3\U(\tau)\phi_i)d\tau}_{H^s}\les
\prod_{i=1}^3\norm{\phi_i}_{H^s}.
\end{eqnarray}
We will construct concrete functions $\phi_i,i=1,2,3$ such that
\eqref{eq:trilinear} fails if $s<1/2$ for any $t>0$.

As in \cite{MR4}, we fix $t\neq 0$ and define the real valued
function $\phi_N$ by:
\[\widehat{\phi_N}(\xi)=N^{-s}\gamma^{-1/2}\big(\chi_{[-\gamma-N,-N]}(\xi)+\chi_{[+N,+N+\gamma]}(\xi)\big),\]
with $\gamma=o(t^{-1})$. Then $\norm{\phi_N}_{H^s}\sim 1$. Let
\[u(x,t)=\int_0^t \U(t-\tau)\partial_x
(\prod_{i=1}^3\U(\tau)\phi_i)d\tau,\] then by straightforward
calculating we have
\begin{eqnarray*}
\ft_x(u)(\xi,t)=i\xi e^{it \dr(\xi)}\int_{\R\times \R}
\frac{e^{itP(\xi,\xi_1,\xi_2)}-1}{iP(\xi,\xi_1,\xi_2)}\widehat{\phi_N}(\xi_1)\widehat{\phi_N}(\xi_2)\widehat{\phi_N}(\xi-\xi_1-\xi_2)d\xi_1d\xi_2,
\end{eqnarray*}
where
\begin{eqnarray*}
P(\xi,\xi_1,\xi_2)=\dr(\xi_1)+\dr(\xi_2)+\dr(\xi-\xi_1-\xi_2)-\dr(\xi).
\end{eqnarray*}
It is easy to see that
\begin{eqnarray*}
&&\ft_x(u)(\xi,t)\chi_{[N-\gamma,N+3\gamma]}(\xi) \simeq i\xi e^{it
\dr(\xi)}\\
&& \cdot \int_{\R\times \R}
\frac{e^{itP(\xi,\xi_1,\xi_2)}-1}{iP(\xi,\xi_1,\xi_2)}\chi_{[N,N+\gamma]}(\xi_1)\chi_{[N,N+\gamma]}(\xi_2)\chi_{[N,N+\gamma]}(\xi-\xi_1-\xi_2)d\xi_1d\xi_2.
\end{eqnarray*}
Since $N\gg 1$ then due to the localization (Note that there is a
cancelation which is crucial)
\begin{eqnarray*}
|P(\xi,\xi_1,\xi_2)|&=&|\coth(\xi_1)+\coth(\xi_2)+\coth(\xi-\xi_1-\xi_2)-\coth(\xi)|\\
&\simeq& |\xi_1^2+\xi_2^2-(\xi-\xi_1-\xi_2)^2-\xi^2|\simeq \gamma^2.
\end{eqnarray*}
Therefore,
\begin{eqnarray}
\norm{u}_{H^s}\ges |t|\gamma N^{-2s}N \ges N^{-2s}N,
\end{eqnarray}
which implies $s\geq 1/2$.

Considering the solution map of Eq. \eqref{eq:mFDF} $\phi
\rightarrow u(t)$, then by computing the Frechet derivatives, we get
\begin{eqnarray*}
\frac{\partial^3 u}{\partial^3 \phi}\big|
_{\phi=0}(h_N,h_N,h_N)=\int_0^t
\U(t-\tau)\partial_x[(\U(\tau)h_N)^3]d\tau.
\end{eqnarray*}
So, if $\phi \rightarrow u$ is of class $C^3$ at the origin, then we
have
\begin{eqnarray}
\normo{\int_0^t \U(t-\tau)\partial_x
[(\U(\tau)h_N)^3]d\tau}_{H^s}\les \norm{h_N}_{H^s}^3,
\end{eqnarray}
which fails as we have showed.

\section{Limit Behavior}

In this section we prove Theorem \ref{limit}. We only prove the
theorem for $s=1/2$ since the other case can be treated in the same
ways. We need the following lemma which follows immediately from the
definition.
\begin{lemma}\label{l81}
Assume $\delta>0$. If $s\in \R$ and $u\in L_t^2H_x^s$, then
\begin{equation}
\norm{u}_{N^s}\les \norm{u}_{L_t^2H_x^s}.
\end{equation}
\end{lemma}
Assume $u_\delta$ is a $H^{1/2}$-strong solution to \eqref{eq:mFDF}
obtained in the last section and v is a $H^{1/2}$-strong solution to
\eqref{eq:mBO} in \cite{Guo}, with initial data $\phi_1,\phi_2\in
H^{1/2}$ satisfying $\norm{\phi_i}_{L^2}\ll 1, i=1,2$, respectively.
From the scaling \eqref{eq:scaling}, we may assume first that
$\norm{\phi_1}_{H^{1/2}},\norm{\phi_2}_{H^{1/2}}\ll 1$. We still
denote by $u_\epsilon, v$ the extension of $u_\epsilon, v$.
 Let
$w=u_\delta-v$ and $\phi=\phi_1-\phi_2$, then $w$ solves
\begin{eqnarray}\label{eq:diff}
\left \{
\begin{array}{l}
\partial_t w-\G_\delta(\partial_x^2 w)+(\G_\delta -\Hl)\partial_x^2 v+(\frac{w(w^2+3u_\delta v)}{3})_x=0, \ t\in \R_{+},\ x\in \R,\\
v(0)=\phi.
\end{array}
\right.
\end{eqnarray}
We first view $(\G_\delta -\Hl)\partial_x^2 v$ as a perturbation to
the difference equation, and consider the integral equation of
\eqref{eq:diff}
\begin{eqnarray*}
w(x,t)=\U(t)\phi-\int_0^t\U(t-\tau)[(\G_\delta -\Hl)\partial_x^2
v+(\frac{w(w^2+3u_\delta v)}{3})_x]d\tau.
\end{eqnarray*}
Then $w$ solves the following integral equation on $t\in [0,1]$,
\begin{eqnarray}
w(x,t)&=&\psi(t)\big[\U(t)\phi-\int_0^t \U(t-\tau)
\chi_{\R_+}(\tau)\psi(\tau)(\G_\delta -\Hl)\partial_x^2
v(\tau)d\tau\nonumber\\
&&\quad -\int_0^t \U(t-\tau)\partial_x[\psi^3(\tau)w(w^2+3u_\delta
v)](\tau)d\tau \big].
\end{eqnarray}
From Lemma \ref{l61} and Lemma \ref{l62}, \ref{l81} and Proposition
\ref{p63}, we get
\begin{eqnarray*}
\norm{w}_{F^{1/2}}\les
\norm{\phi}_{H^{1/2}}+\frac{1}{\delta}\norm{u_\delta}_{L^2_{[0,2]}{H}_x^{3/2}}+\norm{w}_{F^{1/2}}(\norm{v}_{F^{1/2}}+\norm{u_\epsilon}_{F^{1/2}})^2.
\end{eqnarray*}
Since from the proof of Theorem \ref{t11} we have
\[\norm{v}_{F^{1/2}}\les \norm{\phi_2}_{H^{1/2}}\ll 1,\quad \norm{u_\delta}_{F^{1/2}}\les \norm{\phi_1}_{H^{1/2}}\ll 1,\]
then we get that
\begin{equation}
\norm{w}_{F^{1/2}}\les
\norm{\phi}_{H^{1/2}}+\frac{1}{\delta}\norm{u_\delta}_{L^2_{[0,2]}{H}_x^{3/2}}.
\end{equation}
From Lemma \ref{embedding} and Theorem \ref{t11} (d) we get
\begin{eqnarray*}
\norm{u_\delta-v}_{C([0,1], H^{1/2})}\les
\norm{\phi_1-\phi_2}_{H^{1/2}}+\frac{1}{\delta}C(\norm{\phi_1}_{H^{3/2}},\norm{\phi_2}_{H^{1/2}}).
\end{eqnarray*}

For general $\phi_1,\phi_2 \in H^{1/2}$ satisfying
$\norm{\phi_i}_{L^2}\ll 1,\ i=1,2$, using the scaling
\eqref{eq:scaling}, then we immediately get that there exists
$T=T(\norm{\phi_1}_{H^{1/2}},\norm{\phi_2}_{H^{1/2}})>0$ such that
\begin{eqnarray}\label{eq:limitH1half}
\norm{u_\delta-v}_{C([0,T], H^{1/2})}\les
\norm{\phi_1-\phi_2}_{H^{1/2}}+\frac{1}{\delta}C(T,\norm{\phi_1}_{H^{3/2}},\norm{\phi_2}_{H^{1/2}}).
\end{eqnarray}
Therefore, it follows that \eqref{eq:limitH1half} automatically
holds for any $T>0$ due to \eqref{eq:H1halfpriori} and Theorem \ref{t11} (d).\\

\noindent{\bf Proof of Theorem \ref{limit}.} For fixed $T>0$, we
need to prove that $\forall\ \eta>0$, there exists $N>0$ such that
if $\delta>N$ then
\begin{equation}\label{eq:limitH1half2}
\norm{S_T^\delta(\varphi)-S_T(\varphi)}_{C([0,T];H^{1/2})}<\eta.
\end{equation}
We denote $\varphi_K=P_{\leq K}\varphi$. Then we get
\begin{eqnarray*}
&&\norm{S_T^\delta(\varphi)-S_T(\varphi)}_{C([0,T];H^{1/2})}\nonumber\\
&\leq&\norm{S_T^\delta(\varphi)-S_T^\delta(\varphi_K)}_{C([0,T];H^{1/2})}\nonumber\\
&&+\norm{S_T^\delta(\varphi_K)-S_T(\varphi_K)}_{C([0,T];H^{1/2})}+\norm{S_T(\varphi_K)-S_T(\varphi)}_{C([0,T];H^{1/2})}.
\end{eqnarray*}
From Theorem \ref{t11} (d) and \eqref{eq:limitH1half} and the
results in \cite{KenigT, Guo} that the solution map of the modified
Benjamin-Ono equation is Lipschitz continuous, we get
\begin{eqnarray}
\norm{S_T^\delta(\varphi)-S_T(\varphi)}_{C([0,T];H^{1/2})}\les
\norm{\varphi_K-\varphi}_{H^{1/2}}+\frac{1}{\delta}C(T,K,
\norm{\varphi}_{H^{1/2}}).
\end{eqnarray}
We first fix $K$ large enough, then let $\delta$ go to infinity,
therefore \eqref{eq:limitH1half2} holds. \endprf

\noindent{\bf Acknowledgment.} This work is supported in part by the
National Science Foundation of China, grant 10571004; and the 973
Project Foundation of China, grant 2006CB805902, and the Innovation
Group Foundation of NSFC, grant 10621061.

\end{document}